\documentclass[12pt]{article}
\usepackage{amscd}
\usepackage{latexsym,amsmath,amssymb,amsbsy, amsthm}
\usepackage[square]{natbib}
\usepackage{subcaption}

\usepackage{hyperref}
\usepackage{lscape,fancyhdr,fancybox}
\usepackage{graphicx,epsfig, xy}
\usepackage{color}
\usepackage[hmarginratio=1:2, vmarginratio =5:5,
textheight=22cm,bindingoffset=1.5cm, textwidth=14.6cm]{geometry}
\usepackage{tikz}

\newcommand{\comment}[1]{}

\numberwithin{equation}{section}

\newtheorem{remark}{Remark}[section]
\newtheorem{theorem}{Theorem}[section]

\newtheorem{lemma}{Lemma}[section]

\newtheorem{cor}{Corollary}[section]
\theoremstyle{definition}
\newtheorem{definition}{Definition}[section]

\usepackage{amsfonts}

\usepackage{txfonts}

\DeclareMathOperator{\Tr}{Tr}

\newcommand{\beq}{\begin{eqnarray}}
\newcommand{\eeq}{\end{eqnarray}}
\newcommand{\ben}{\begin{eqnarray*}}
\newcommand{\een}{\end{eqnarray*}}

\allowdisplaybreaks

\begin{document}
\def\shorttitle{Covariance matrix}
 \def\shortauthors{A. Bose, P. Sen}
\title{\textbf{\Large \sc
\Large{Sample covariance matrices and special symmetric partitions} 
}}\small

\author{
%\hspace{.05\textwidth}
 \parbox[t]{0.20
\textwidth}{{\sc Arup Bose}
 \thanks{Statistics and Mathematics Unit, Indian Statistical Institute, 203 B.T. Road, Kolkata 700108, India. email: bosearu@gmail.com. }}
%Research  supported by J.C. Bose National Fellowship, Department of Science and Technology, Govt. of India.}}
%% \\ {\small Stat-Math Unit\\ Indian Statistical Institute\\ 203 B.T. Road,\\  Kolkata 700108\\ INDIA}\\}
\parbox[t]{0.25\textwidth}{{\sc Priyanka Sen}
 \thanks{Statistics and Mathematics  Unit, Indian Statistical Institute, 203 B.T. Road, Kolkata 700108, INDIA. email: priyankasen0702@gmail.com}}
}

\date{\today}  
\maketitle

\begin{abstract}
Suppose $X_p$ is a real $p \times n$ matrix with independent entries and consider the (unscaled) sample covariance matrix $S_p=X_pX_p^T$. The Mar$\check{\text{c}}$enko-Pastur law was discovered as the limit of the bulk distribution of the sample covariance matrix in 1967. There has been extensions of this result in several directions. In this paper we consider an extension that handles several of the existing ones as well as generates new results.

  We show that under suitable assumptions on the entries of $X_p$, the limiting spectral distribution exists in probability or almost surely. The moments are described by a set of partitions that are beyond pair partitions and non-crossing partitions and are also related to special symmetric partitions, which are known to appear in the limiting spectral distribution of Wigner-type matrices. Similar results hold for other patterned matrices such as reverse circulant, circulant, Toeplitz and Hankel matrices. 
\end{abstract}

\vskip 5pt
\noindent \textbf{Key words and phrases.} Empirical and limiting spectral distribution, Wigner matrix, Sample Covariance matrix,  Mar$\check{\text{c}}$enko-Pastur law, non-crossing partition, special symmetric partition, cumulant, free cumulant, exploding moments, hypergraphs, sparse matrices.  

\medskip

\section{Introduction}\label{introduction}
 Suppose $A_n$ is an $n\times n$ real symmetric random matrix with (real) eigenvalues $\lambda_1,\lambda_2,\ldots,\lambda_n$. Its \textit{empirical spectral distribution} or measure (ESD) is the random probability measure:
\begin{align*}
\mu_{A_n} = \frac{1}{n} \sum_{i=1}^{n}  \delta_{\lambda_i},
\end{align*}
where $\delta_x$ is the Dirac measure at $x$. The expected empirical spectral distribution (EESD) is the corresponding expected measure. 
%We shall assume that $n \to \infty$. 

The notions of convergence used for the ESD are the weak convergence of the EESD and the weak convergence of the ESD, either in probability or almost surely. The limits are identical when the latter two limits are non-random. In any case, any of these limits will be referred to as the \textit{limiting spectral distribution} (LSD) of $\{A_n\}$. 

Consider the $p \times n$ matrix $X_p$ with real independent entries $\{x_{ij,n}: 1\leq i \leq p, 1 \leq j \leq n\}$. We are interested in the LSD of the matrix $S_p=X_pX_p^{T}$ when $p=p(n), p/n \rightarrow y\in (0,\infty)$. This matrix will be called the \textit{Sample covariance matrix} (without scaling) and is denoted by $S$. It is usual to assume that the dimension $p$ tends to $\infty$ proportionally to the degrees of freedom $n$.

There have been several works regarding the LSD of these matrices.
The LSD of $\frac{1}{n}S_p$, where the entries of $X_p$ are iid with mean zero and fourth moment finite, was first established in \citep{marchenko1967distribution} and is called the Mar$\check{\text{c}}$enko-Pastur (M-P) law after them. Subsequent works followed in \citep{grenander1977spectral}, \citep{wachter1978strong}, \citep{yin1986limiting},  where the authors investigated the LSD under varied assumptions on the entries. For instance, in  \citep{wachter1978strong}, the author proves the M-P law under the assumption that the entries are independent with common mean and variance and $(2+\delta)$th moment finite.  \citep{belinschi2009spectral} proved the convergence of  the ESD of $\frac{1}{a_n^2}S_p$ when the entries of $X_p$ are iid with heavy tails and $a_n$ is a sequence of constants related to the tail probability of the entry distribution. An appropriate truncation of the variables at levels that depend on $n$ was crucial in their arguments. It thus becomes relevant to probe the case where the distribution of the entries is also allowed to depend on $n$. Such a model was considered by \citep{zakharevich2006generalization} for the Wigner matrix. She studied the LSD of a generalized Wigner matrix $W_n$ whose entries $\{x_{ij,n}; 1 \leq i \leq j \leq n\}$ are iid with distribution $G_n$ for every fixed $n$. Assume that, $$\displaystyle{\lim_{n\to\infty} \frac{\mu_n(k)}{n^{k/2-1}\mu_n(2)^{k/2}}}= g_k,  \ \text{say,  exists for all} \ \ k\geq 1$$ where $\mu_n(k)$ is the $k$th moment of $G_n$. Then she proved that the ESD of $\frac{W_n}{\sqrt{n\mu_n(2)}}$ converges in probability to a distribution $\mu_{zak}$ that depends only on the sequence $\{g_{2k}\}$. \citep{bose2021random} explored Wigner matrices with general independent triangular array of entries and found that a class of partitions, which they called special partitions, play a crucial role in the limiting moments. They also discovered that there is a one-to-one correspondence between the rooted trees given in \citep{zakharevich2006generalization} and the special symmetric partitions.

 In general such matrices are referred to as \textit{matrices with exploding moments} and have been considerd by several authors.  In particular sample covariance matrices with exploding moments have been studied in \citep{Benaych-Georges2012} (Theorem 3.2) and \citep{noiry2018spectral}(Proposition 3.1). These articles have given formulae for the moments of the LSD using techniques in free probability theory and graphs respectively. However to the best of our knowledge, direct links between the two types of moments formulae have not be explored. We describe a moment formula using partitions that relates these moments to the limiting moments in the Wigner case. It turns out that again only the class of special symmetric partitions  contribute to the moments. We also study the LSD of covariance matrices with sparse entries and find a relation of the LSD with free Poisson as well as Poisson variables.

 In \citep{bose2020some}, the ESD of symmetric patterned matrices such as symmetric reverse circulant, symmetric circulant, symmetric Toeplitz and symmetric Hankel matrices have been dealt with where the entries of the matrices are independent and satisfy certain moment conditions. With similar assumptions on the entries of $X_p$, where it is one of the four matrices mentioned, it can be shown that the ESD of $X_pX_p^T$ converges almost surely to non-random real probability distributions. Further, $X_p$ could be chosen to be the non-symmetric versions of these patterned matrices and the ESD of $X_pX_p^T$ can be shown to converge almost surely to non-random real probability distributions. We shall deal with all these results in a subsequent article. For a brief overview of which partition sets play crucial role in describing the limiting moment of these matrices, see Remark \ref{other patterns}.
 
\section{Main results}\label{main results}
The notion of \textit{multiplicative extension} is required to describe our results. Let $[k]:=\{1,2,\ldots, k\}$ and let $\mathcal{P}(k)$ denote the set of all partitions of $[k]$. Let $\mathcal{P}_{2}(2k)$ be the set of \textit{pair-partitions} of $[2k]$.   
Suppose $\{c_k, \ k \geq 1\}$ is any sequence of numbers. Its multiplicative extension is defined on $\mathcal{P}_k$, $k \geq 1$ as follows. For any $\sigma\in \mathcal{P}_k$,  define 
\begin{equation*}c_{\sigma}=\prod_{V \ \text{is a block of}\ \sigma} c_{|V|}.
\end{equation*}
We also need the following notion of \textit{special symmetric partitions} which were introduced in \citep{bose2021random}.
% (see Definition 2.1). These partitions have already appeared in \citep{bose2021random}.  
\begin{definition}(\textit{Special Symmetric Partition})\label{ss(2k)}
Let $\sigma \in \mathcal{P}(k)$ and let $V_1, V_2, \ldots $ be the blocks of $\sigma$, arranged in ascending order of their smallest elements.  This partition is said to be \textit{special symmetric} if 
\begin{enumerate}
\item[(i)] The last block is a union of sets of even sizes 
of consecutive integers.
\item[(ii)] Between any two successive elements of any block there are even number of elements from any other  block, and they each appear equal number of times in an odd and an even position.
\end{enumerate}
\end{definition}
We denote the set of all special symmetric partitions of $[k]$ by $SS(k)$ and the set of all special symmetric partitions of $[k]$ with $b$ distinct blocks by $SS_b(k)$. Clearly $SS(k)=\emptyset$ when $k$ is odd. 

Let $SS_b(2k)\subset SS(2k)$ where each partition has exactly $b$ blocks.  Clearly $b \leq k$ always. The one-block partition $\{1, 2, \ldots , 2k\}\in SS(2k)$. Each block of $SS(2k)$ has even size. There are $\pi \in SS(2k)$ that are either crossing or not paired. For example the partition  $\pi=\{\{1, 2, 5, 6\}, \{3, 4, 7, 8\}\}\in SS(8)$ but is crossing. On the other hand, $\sigma=\{\{1, 2, 6, 7\}, \{3, 4, 5, 8\}\}\notin SS(8)$.\\ 
It is easy to check that 
$$SS(2k)\cap \mathcal{P}_{2}(2k)=NC_2(2k) \subset SS(2k).$$ 
\vskip3pt

Now we introduce a set of assumptions on the entries $\{x_{ij,n}\}$ of $X_p$. We drop the suffix $n$ for convenience wherever there is no scope for confusion. For any real-valued function $g$ on $[0, \ 1]$,  $\|g\|:=\sup_{0 \leq x\leq 1} |g(x)|$ will denote its sup norm. \\

\noindent \textbf{Assumption A}. 
Suppose $\{g_{2k,n}\}$ is a sequence of bounded Riemann integrable functions on $[0,1] ^2$. There exists a sequence $\{t_n\}$ with $t_n\in [0,\infty]$ such that 
\begin{enumerate}
\item [(i)] For each $k \in \mathbb{N}$,
\begin{align}
 & n \ \mathbb{E}\left[x_{ij}^{2k}\boldsymbol {1}_{\{|x_{ij}|\leq t_n\}}\right]= g_{2k,n}\Big(\frac{i}{p},\frac{j}{n}\Big) \ \ \ \text{for } \ 1\leq i\leq p, 1 \leq j \leq n, \label{gkeven}\\
& \displaystyle \lim_{n \rightarrow \infty} \ n^{\alpha} \underset{1\leq i \leq p,1 \leq j \leq n}{\sup} \ \mathbb{E}\left[x_{ij}^{2k-1}\boldsymbol {1}_{\{|x_{ij}|\leq t_n\}}\right] = 0  \ \ \text{for any } \alpha<1 . \label{gkodd}
\end{align}
\item [(ii)] The functions $g_{2k,n}(\cdot), n \geq 1$ converges uniformly to $g_{2k}(\cdot)$ for each $k \geq 1$. 
\item[(iii)] Let  $M_{2k}=\|g_{2k}\|$, $M_{2k-1}=0$ for all $k \geq 1$. Then $\alpha_{2k}=  \sum_{\sigma \in \mathcal{P}(2k)} M_{\sigma}$  satisfies \textit{Carleman's condition}, 
$$ \displaystyle \sum_{k=1}^{\infty} \alpha_{2k}^{-\frac{1}{2k}}= \infty.$$
\end{enumerate}

\begin{theorem}\label{res:XXt}
Let $X_p$ be a $p \times n$ real matrix with independent entries $\{x_{ij,n}; 1\leq i\leq p, 1\leq j \leq n\}$ that satisfy Assumption A and $p/n \rightarrow y\in(0,\infty)$ as $n \rightarrow \infty$. Suppose $Z_p$ is a $p \times n$ real matrix whose entries are $y_{ij}=x_{ij}\boldsymbol {1}_{[|x_{ij}|\leq t_n]}$. Then
\begin{enumerate}
\item[(a)] The ESD of $Z_pZ_p^T$ converges  weakly almost surely to a probability measure $\mu$ say, whose moments are determined by the functions $g_{2k}, k\geq 1$ as in \eqref{moment-XXt}.
\item[(b)] Moreover, if $\frac{1}{n} \displaystyle\sum_{i=1,j=1}^{p,n} \ x_{ij}^2\boldsymbol {1}_{\{|x_{ij}| > t_n\}} \rightarrow 0, \ \mbox{almost surely (or in probability)}$, then the ESD of $S_p=X_pX_p^T$ converges weakly almost surely (or in probability) to the probability measure $\mu$ given in (a).
\end{enumerate} 
\end{theorem}
\begin{remark}\label{unbounded support}
Suppose the entries of $X_p$ satisfy Assumption A. Let for every $m\geq 1$, $f_{2m}(x)=\int_{[0,1]}g_{2m}(x,y)\ dy$. Now suppose there exist $m>1$ such that $ \displaystyle \inf_{t\geq 1}\int_{[0,1]}\bigg(\frac{f_{2m}(x)}{m!}\bigg)^t \ dx = c >0$. Then the LSD $\mu$ in Theorem \ref{res:XXt} has unbounded support. 
\end{remark}
\begin{cor}\label{p=n}
Suppose $p=n$ and the entries of the matrix $X_p$ satisfy Assumption A. Then from Theorem 2.1 in \citep{bose2021random} we know that the ESD of $W_n$, the Wigner matrix (i.e., symmetric matrix with independent entries $\{x_{ij,n}; 1\leq i \leq j \leq n\}$) converges almost surely to a symmetric probability measure $\mu^{\prime}$. Suppose $X$ and $Y$ are two random variables such that $X \sim \mu$ and $Y \sim \mu^{\prime}$. If $\{g_{2k}\}_{k\geq 1}$ are symmetric functions, then, $X \overset{\mathcal{D}}{=} Y^2$.
\end{cor}

\begin{remark}\label{other patterns}
Consider $X_p$ to be of the form of other patterned matrices such as reverse cicrulant, circulant, Toeplitz and Hankel. The following link functions describe these matrices:
\begin{itemize}
\item[(i)] Symmetric reverse circulant ($R_n^{(s)}$): $L(i,j)= (i+j-2)(\text{mod }n), 1 \leq, i,j \leq n$.
\item[(ii)] Non-symmetric reverse circulant ($R_p$): $L(i,j)= \begin{cases}
(i+j-2)(\text{mod }n) & i\leq j,\\
-[(i+j-2)(\text{mod }n) ] & i> j.
\end{cases}$ 
\item[(iii)] Symmetric circulant ($SC_n$): $L(i,j)= n/2-|n/2-|i-j||, 1 \leq, i,j \leq n$.
\item[(iv)] Circulant ($C_p$): $L(i,j)=(j-i) (\text{mod }n)$.
\item[(v)] Symmetric Toeplitz ($T_n^{(s)}$): $L(i,j)= |i-j|, 1 \leq, i,j \leq n$.
\item[(vi)] Non-symmetric Toeplitz ($T_p$): $L(i,j)= i-j$.
\item[(vii)] Symmetric Hankel ($H_n^{(s)}$): $L(i,j)= i+j$.
\item[(viii)] Non-symmetric Hankel ($H_p$): $L(i,j)= \begin{cases}
(i+j) & i\geq j,\\
-(i+j) & i< j.
\end{cases}$ 
\end{itemize}
Suppose the entries satisfy similar conditions as Assumption A. Then the ESD of $X_pX_p^T$ can also be shown to converge weakly almost surely (or in probability) to non-random probability measures. We will be able to show that the partitions that contribute to the limiting moments for these matrices include symmetric partitions and even partitions that have also appeared while dealing with the ESD of independent patterned matrices (when $p=n$) in \citep{bose2020some}.
For the matrices (i), (ii), (iv), (vi), (vii), (viii) the partitions that describe the limiting moments are symmetric partitions. However their contribution in each of the cases might differ. On the other hand, for (iii) and (v), the limiting moments can be described by even partitions and in these cases too their contributions might differ.
 % If $X_p$ is the reverse circulant (symmetric or non-symmetric) or Hankel matrix (symmetric or non-symmetric) or non-symmetric Toeplitz (with link function $L(i,j)=i-j$) or Circulant (with link function $L(i,j)=(j-i)\ (\text{mod } n)$), then the partitions that contribute to the limiting moments of the ESD of $X_pX_p^T$ are symmetric. On the other hand if $X_p$ is symmetric Toeplitz or symmetric circulant, then the the partitions that contribute to the limiting moments of the ESD of $X_pX_p^T$ are even.  
  We shall deal with the ESD of $X_pX_p^T$ for all these patterned matrices in details in a subsequent article.
\end{remark}
\section{Relation to existing results}\label{realtion to existing results} 

\subsection{ \textbf{IID entries}}\label{iid}
 Suppose the entries of $X_p$ are $\{x_{ij}/\sqrt{n}\}$ where $\{x_{ij}\}$ are i.i.d. with distribution  $F$ which has mean zero and variance $1$.  
 
 Let $t_n=n^{-1/3}$. Then $g_2\equiv 1$ and $g_{2k}\equiv 0, k>1$, follows using the same argument as Section 5.1 (a) in \citep{bose2021random}.
% Then $ t_n\sqrt{n} \rightarrow \infty$ as $n \rightarrow \infty$ and
%\begin{align*}
%\displaystyle \lim_{n \rightarrow \infty} n \ \mathbb{E}\bigg [\bigg(\frac{x_{ij}}{\sqrt{n}}\bigg)^2 \boldsymbol{1}_{[|x_{ij}/\sqrt{n}| \leq t_n]}\bigg] \ =  \ 1=  C_2.
%\end{align*}
%Also, for any $k>2$,
%\begin{align*}
%n\ \mathbb{E}\bigg [\bigg(\frac{x_{ij}}{\sqrt{n}}\bigg)^{k} \boldsymbol{1}_{[|x_{ij}/\sqrt{n}| \leq t_n]}\bigg] \ &
%= n\ \mathbb{E}\big [{(x_{11}/\sqrt{n})}^{(k-2)} \ {(x_{11}/\sqrt{n})}^{2} \boldsymbol{1}_{[|x_{11}| \leq t_n  \sqrt{n}]}\big]\\
%& \leq n \frac{t_n^{(k-2)}}{n} \mathbb{E} \big [{x_{11}}^{2} \boldsymbol{1}_{[|x_{11}| \leq t_n  \sqrt{n}]}\big]\\
%& \leq t_n^{(k-2)} =  (n^{- \frac{1}{3}})^{k-2} \ \rightarrow \ 0 \ \ \text{ as } n \rightarrow \infty.
%\end{align*}
Now for any $ t > 0$, 
\begin{align*}
\frac{1}{p}\displaystyle \sum_{i,j} \ \big(x_{ij}/\sqrt{n}\big)^2[\boldsymbol {1}_{[|x_{ij}/\sqrt{n}| > t_n]}]= & \frac{1}{np}\displaystyle \sum_{i,j} \ x_{ij}^2[\boldsymbol {1}_{[|x_{ij}| > t_n \sqrt{n}]}]\\
& \leq \frac{1}{np}\displaystyle \sum_{i,j} \ x_{ij}^2[\boldsymbol {1}_{[|x_{ij}| > t]}] \ \ \text{for all large} \ n,\\
& \overset{a.s.}{\longrightarrow} \ \mathbb{E}\big[ x_{11}^2[\boldsymbol {1}_{[|x_{11}| > t]}]\big] \ \ \text{for all large} \ \ n.  
\end{align*}
As $\mathbb{E}[x_{11}^2]=1$, taking $t$ to infinity, the above limit is 0 almost surely. 

%So without loss we may assume that all moments of $F$ are finite.
%Let $G_n$ be the distribution of $X/ \sqrt{n}$ for each $n$ where $X\sim F$.
%So the $k$th moment of $G_n$ equals $\mu_n(k)= \frac{\beta_k(F)}{n^{k/2}}$ for $k\geq 1$. Thus $n\mu_n(2)=\beta_2(F)=1$  for all $n$. Also, for $k>2$, $n\mu_n(k)=\frac{\beta_k(F)}{n^{k/2-1}}$. As $F$ has all moments finite, we have  $g_2=1$ and $g_{2k}=0$ for all $k>1$. 
 Hence applying Theorem \ref{res:XXt}, the ESD of $S_p$ converges weakly almost surely to $\mu$ whose $k-$th moment is given by 
\begin{align}
\beta_k(\mu)&= \displaystyle \sum_{r=0}^{k-1} \sum_{\underset{\underset{\text{ even generating vertices}}{\text{having }(r+1)}}{\pi \in SS_k(2k)}} y^r 
\end{align}
Incidentally, generating vertices and even generating vertices of partitions or words are common concepts while using the moment method. We describe these notions in details later in Section \ref{preliminaries}.
 
\noindent Now the number of pair matched words of length $2k$ with $r+1$ even generating vertices is shown to be $\frac{1}{r+1}{k \choose r}{{k-1} \choose r}$ in Theorem 5(a) in \citep{bose2008another}. Hence the rhs of the above equation reduces to the $k$th moment of the $MP_y$ law. Hence we obtain that the ESD of $\frac{1}{n}S_p$ converges to the $MP_y$ law when the entries are iid with mean 0 and variance 1.

\subsection{ \textbf{Heavy-tailed entries}} Suppose $\{x_{ij}, 1 \leq i \leq p, 1 \leq j\leq n\}$ are i.i.d. with an $\alpha$-stable distribution ($0<\alpha <2$) and $n/p \rightarrow \gamma \in (0,1]$. Then 
$$\mathbb{P}[|x_{ij}|\geq u]= \frac{L(u)}{u^{\alpha}}, \   u > 0$$ 
where $L$ is a slowly varying function at $\infty$.  
Let $X_p=(x_{ij}/a_p)$ and $$a_p= \inf\{u: \mathbb{P}[|x_{ij}|\geq u]\leq \frac{1}{p}\}.$$ \citep{belinschi2009spectral} proved the existence of LSD of $S_p$  using the method of Stieltjes transform. We show how our theorem  may be used to  give an alternative proof. 
 
For a fixed constant say, $B$, consider the matrix $X_p^B$ with entries are $\frac{x_{ij}}{a_p}\boldsymbol {1}_{[|x_{ij}|\leq B a_p]}$. Then for every fixed $B\in \mathbb{N}$, $X_p^B$ satisfies Assumption A. Hence from Theorem \ref{res:XXt}, there exists a probability measure $\mu_B$ which is the weak limit of the ESD of $S_p^B$, almost surely. Next recall the arguments used for heavy-tailed Wigner matrices in Section 5.2 of \citep{bose2021random}. These arguments can be adapted appropriately for the $S-$matrix. Thus it can be seen that  that $\mu_{S_p}$ converges weakly to $\tilde{\mu}$ in probability. This yields the convergence of Theorem 1.10 in \citep{belinschi2009spectral}.

\subsection{ \textbf{Triangular iid (size dependent matrices)}}\label{triangular iid}
 Suppose $\{x_{ij,n}; 1\leq i \leq p,1 \leq j \leq n \}$ is a sequence of iid random variables with distribution $F_n$ that has finite moments of all orders, for every $n$. Also assume that for every $k \geq 1$,
\begin{align}\label{ck}
n \beta_k(F_n) \rightarrow C_k
\end{align}
where $\beta_k(F_n)$ denotes the $k$th moment of $F_n$. Suppose $\{0,C_2,0,C_4, \ldots \}$ is the cumulant sequence of a probability distribution whose moment sequence satisfies Carleman's condition. 

Let the entries of $X_p$ be  $\{x_{ij,n}; 1\leq i \leq p,1 \leq j \leq n \}$ as described above. Then from Theorem \ref{res:XXt}, it follows that as $p/n \rightarrow y>0$ the ESD of $S_p$ converges to a non-random probability distribution $\mu$ whose moment sequence is determined by $\{C_k\}_{k \geq 1}$ and $y$. Indeed, from \eqref{ck} it follows that Assumption A is true with $t_n=\infty$ and $g_{2k}\equiv C_{2k}$ for all $k \geq 1$. Therefore from Theorem \ref{res:XXt} we find that there is a  probability measure $\mu$ such that the ESD of $S_p$ converges to $\mu$ and its  moments are given as follows:
\begin{align}\label{Zak-limit}
\beta_{k}(\mu)=\displaystyle \sum_{r=0}^{k-1}\sum_{\underset{\underset{\text{ even generating vertices}}{\text{having }(r+1)}}{\pi \in SS(2k)}} y^r C_{\pi}.
\end{align}

Now we consider the set up of  \citep{zakharevich2006generalization}
for $X_p$, i.e, the entries of $X_p$ satisfy the conditions of Theorem 3.2 in \citep{Benaych-Georges2012} and Proposition 3.1 in \citep{noiry2018spectral}. Clearly Assumption A is satisfied with $t_n=\infty$, $g_2\equiv 1$ and $g_{2k} \equiv C_{2k}, k \geq 2$. Hence Theorem \ref{res:XXt} can be applied and the resulting LSD, say, $\mu_{bcn}$ has moments as in \eqref{Zak-limit}. In Section \ref{connection-ss2k}, we shall verify that this LSD is same as obtained in the above references.
%\citep{Benaych-Georges2012} and \citep{noiry2018spectral}.

\vskip3pt

\noindent \textbf{Connection to the limiting moments of the Wigner matrices}: As observed in the previous paragraphs, when the entries are triangular iid that satisfy \eqref{ck}, $g_{2k}\equiv C_{2k}, k \geq 1$ and thus symmetric. Therefore from Corollary \ref{p=n}, we have that $X \overset{\mathcal{D}}{=} Y^2$ where $X \sim \mu$, $Y \sim \mu^{\prime}$ , $\mu$ and $\mu^{\prime}$ are the LSDs of the covariance matrix and the Wigner matrix respectively.

\subsubsection{ \textbf{Sparse S}} Suppose the entries of $X_p$ are i.i.d. $Ber(p_n)$ for each $n$, with $np_n \rightarrow \lambda>0$. Then  the entries satisfy \eqref{ck} with $C_{k}\equiv \lambda$ for all $k\geq 1$. Hence by Theorem \ref{res:XXt} as discussed in the beginning of this section, the ESD of $S_p$ converges weakly almost surely to $\mu$ whose moments are as follows (see \eqref{Zak-limit}):
\begin{equation}\label{sparse-hom}
\beta_{k}(\mu)= \displaystyle \sum_{r=0}^{k-1}\sum_{\underset{\underset{\text{ even generating vertices}}{\text{having }(r+1)}}{\pi \in SS(2k)}} y^r \lambda^{|\pi|}.
\end{equation}

Let $E(2k)$ (and $NCE(2k)$) be the set of partitions (and non-crossing partitions) whose blocks are all of even sizes. Then from Definition \ref{ss(2k)}, we have $$NCE(2k) \subset SS(2k) \subset E(2k).$$ Therefore we have the following:\\
\vskip3pt

\noindent Case 1: $y \leq 1$. Then from \eqref{sparse-hom}
\begin{equation}\label{y<=1}
\displaystyle \sum_{\pi \in NCE(2k)} (\lambda y)^{|\pi|} < \beta_{k}(\mu) < \sum_{\pi \in E(2k)} \lambda^{|\pi|}
\end{equation}
Case 2: $y > 1$. Then from \eqref{sparse-hom}
\begin{equation}\label{y>1}
\displaystyle \sum_{\pi \in NCE(2k)} y^{|\pi|} < \beta_{k}(\mu) < \sum_{\pi \in E(2k)} (\lambda y)^{|\pi|}
\end{equation}

Now suppose $P_1(\gamma)$ is a free Poisson variable with mean $\gamma$ and $P_2(\gamma)$  is Poisson with mean $\gamma$. Let $Y$ be a random variable which takes value $1$ and $-1$ with probability $\frac{1}{2}$ each. Suppose $Y$ is independent of $P_1(\gamma)$ and $P_2(\gamma)$. Consider $Q_1(\gamma)= P_1(\gamma)Y$ and $Q_2(\gamma)= P_2(\gamma)Y$. Then the moments of $Q_1(\gamma)$ and $Q_2(\gamma)$ are give as follows:
\begin{align}
\mathbb{E}[Q_1^k(\gamma)]&=\begin{cases}
0 \ & \ \text{ if } k \text{ is odd },\\
\displaystyle \sum_{\pi \in NCE(k)} \gamma^{|\pi|} & \ \text{ if } k \text{ is even }. 
\end{cases} \label{Q_1}\\
\mathbb{E}[Q_2^k(\gamma)]& =\begin{cases}
0 \ & \ \text{ if } k \text{ is odd },\\
\displaystyle \sum_{\pi \in E(k)} \gamma^{|\pi|} & \ \text{ if } k \text{ is even }. 
\end{cases}\label{Q_2}
\end{align}

Hence from \eqref{y<=1} and \eqref{y>1}, we have that if $X$ is a random variable such that $X \sim \mu$, then $X$ is dominated below and above by $(Q_1(\lambda y))^2$ and $(Q_2(\lambda))^2$ when $y\leq 1$ and by $(Q_1(\lambda ))^2$ and $(Q_2(\lambda y))^2$ when $y>  1$ in the sense that
\begin{align*}
\mathbb{E}[(Q_1(\lambda y))^{2k}]< \mathbb{E}[X^k]< \mathbb{E}[(Q_2(\lambda ))^{2k}] \ \ \text{ for every } k\geq 1, y\leq 1,\\
\mathbb{E}[(Q_1(\lambda))^{2k}]< \mathbb{E}[X^k]< \mathbb{E}[(Q_2(\lambda y))^{2k}] \ \ \text{ for every } k\geq 1, y> 1.
\end{align*} 

\subsubsection{ \textbf{Matrices with a variance profile}}
Consider the matrix $X_p$ with a variance profile $\sigma$, i.e. suppose the entries of the matrix $X_p$ are $\{y_{ij,n}=\sigma(i/p,j/n)x_{ij,n}; 1\leq i \leq p, 1 \leq j\leq n \}$ where $\{x_{ij,n};1\leq i \leq p, 1 \leq j\leq n \}$ are i.i.d. for every fixed $n$ and satisfy the two conditions given in \eqref{ck} and $\sigma$ is a bounded piecewise continuous function on $[0,1]^2$. 
 Then the ESD of $S_p$ converges weakly almost surely to a symmetric probability measure $\nu$ whose $k$th moment is determined by $\sigma$ and $\{C_{2m}\}_{1\leq m\leq 2k}$.

Indeed, observe that $y_{ij,n}$ satisfy Assumption A with $g_{2k}\equiv \sigma^{2k}C_{2k}$. 
 Thus using Theorem \ref{res:XXt}, we conclude that the ESD of $S_p$ converges weakly almost surely to a probability measure $\nu$. 

 We shall give a description of the moments for the limit distribution, $\nu$ after the proof of Theorem \ref{res:XXt} as it is quite involved.  It is evident from the moment formula that for two distinct special symmetric partitions $\pi_1$ and $\pi_2$ with the same number of blocks and block sizes, the contribution in the limiting moments can differ. %(which is not always the case for iid entries). 

Suppose that $p=n$ and the entries of $X_p$ are $\{y_{ij}/\sqrt{n}; 1\leq i,j \leq n\}$ where $y_{ij}= \sigma(i/n,j/n)x_{ij}$ with $\sigma : [0,1]^2 \longrightarrow [0,1]$, $\sigma(x,y)= 1, x \leq y$ and 0 otherwise and $\{x_{ij}; 1\leq i,j \leq n \}$ are iid with mean zero and variance 1. Then a truncation argument similar to Section \ref{iid}, can be used to make the reduction that $\{x_{ij}\}$ has all moments finite. Hence from the discussion above we can conclude that  the ESD of $TT^{*}$ where $T$ is the triangular matrix described in Lemma 8.4 in \citep{Dykema2002DToperatorsAD} converges weakly almost surely to a non-random probability measure.

\section{Link function, Circuits and Words}\label{preliminaries}
\noindent Link functions, circuits and words have been described in Section 3 of \citep{bose2021random}. In this section, we describe these concepts particularly for the sample covariance matrix.

%Let 
%\begin{align*}
%L_n:\{1,2,\ldots ,n\}^2 \longrightarrow \mathbb{Z}^d, \ \,n \geq 1, d=1 \text{ or } 2
%\end{align*}
%be a sequence of functions, called \textit{link functions}. Often we write $L_n=L$ for convenience. Most common patterned matrices can be represented using link functions. For instance, the Wigner matrix can be described as $W_n= (x_{L(i,j)})$ where $L(i,j)= (\min(i,j), \max(i,j))$. 
\noindent \textbf{Link function}:  The relevant link function for the sample covariance matrix is given by a pair of functions as follows:
\begin{align*}
L1(i,j)=(i,j) \ \ \text{ and } \ \ L2(i,j)=(j,i).
\end{align*}
We shall explain the relevance of these link functions for the covariance matrix below.\\
%when we describe \textit{circuits} in the next paragraph. 

%\noindent A \textbf{circuit} $\pi$ is a function $\pi : \{0,1,2,\ldots,k\}\longrightarrow \{1,2,3,\ldots,n\}$ with $\pi(0)=\pi(k)$. We say that the \textit{length} of  $\pi$ is $k$ and denote it by $\ell(\pi)$. If $A_n=\big(x_{L(i,j)}\big)$, then using  circuits, we can express the trace, $\Tr(A_n^k)$ as
%\begin{align}\label{momentf1usual}
%%\mathbb{E}\big[\Tr(A_{n}^{k})\big]  = \sum_{\pi:\ell(\pi)=k}\mathbb{E}\big[x_{L(\pi(0),\pi(1))}x_{L(\pi(1),\pi(2))}\cdots x_{L(\pi(k-1),\pi(k))}\big]=\sum_{\pi:\ell(\pi)=k}\mathbb{E}[X_{\pi}],
%\Tr(A_{n}^{k})  = \sum_{\pi:\ell(\pi)=k}x_{L(\pi(0),\pi(1))}x_{L(\pi(1),\pi(2))}\cdots x_{L(\pi(k-1),\pi(k))}=\sum_{\pi:\ell(\pi)=k}X_{\pi},
%\end{align}
%where $X_{\pi}= x_{L(\pi(0),\pi(1))}x_{L(\pi(1),\pi(2))}\cdots x_{L(\pi(k-1),\pi(k))}$. 

\noindent \textbf{Circuits}: For the matrix $S_p$, observe that, for a circuit $\pi$ of length $2k$, $\pi(0)=\pi(2k); 1 \leq \pi(2i)\leq p, \forall 1 \leq i \leq k; 1 \leq \pi(2i-1)\leq n, \forall 1 \leq i \leq k $. Next let
\begin{align*}
\xi_{\pi}(2i-1) &= L1(\pi(2i-2),\pi(2i-1)), 1 \leq i \leq k\\
\xi_{\pi}(2i) &= L2(\pi(2i-1), \pi(2i)),1 \leq i \leq k.
\end{align*}
Then,
\begin{align}\label{genmoment-XXt}
\mathbb{E}\big[\Tr(S_{p}^{k})\big]& =\mathbb{E}\big[\Tr(X_{p}X_p^*)^{k}\big]\nonumber \\
 & =\sum_{\pi:\ell(\pi)=k}x_{L1(\pi(0),\pi(1))}x_{L2(\pi(1),\pi(2))}\cdots x_{L2(\pi(2k-1),\pi(2k))}\nonumber\\
 &= \sum_{\pi:\ell(\pi)=2k}\mathbb{E}[Y_{\pi}],
\end{align}
where $Y_{\pi}= \displaystyle \prod_{i=1}^k y_{\xi_{\pi(2i-1)}}y_{\xi_{\pi(2i)}}$.\\

For any $\pi$, the values $Lt(\pi(i-1),\pi(i)),t=1,2$ will be called \textit{edges}. When an edge appears more than once in a circuit $\pi$, then it is called 
%$L$-
\textit{matched}. Any $m$ circuits $\pi_1,\pi_2,\ldots,\pi_m$ are said to be \textit{jointly-matched} if each edge occurs at least twice across all circuits. They are said to be \textit{cross-matched} if each circuit has an edge which occurs in at least one of the other circuits. \\

% For any link function $L$, circuits $\pi_1$ and $\pi_2$ are \textit{equivalent} if and only if for all $1 \leq i,j \leq k$,
%\begin{align*}
%L(\pi_1(i-1),\pi_1(i))=L(\pi_1(j-1),\pi_1(j))
%\Longleftrightarrow L(\pi_2(i-1),\pi_2(i))=L(\pi_2(j-1),\pi_2(j)).
%\end{align*} 
%The above is an equivalence relation on $\{\pi:\ell(\pi)=k\}$ for any given $k\geq 1$. Any equivalence class  of circuits can be indexed by an element of $\mathcal{P}(k)$. The positions where the edges match are identified by each block of a partition of $[k]$. For example, the partition $\{\{1,3\}, \{2,4,5\}\}$ of $[5]$ corresponds to the equivalent class 
%$$\big\{\pi:\ell(\pi)=5\ \mbox{and}\ L(\pi(0),\pi(1))=L(\pi(2),\pi(3)),\ L(\pi(1),\pi(2))=L(\pi(3),\pi(4))=L(\pi(4),\pi(5))\big\}.$$  
%Also, an element of $\mathcal{P}(k)$ can be identified with a \textit{word} of length $k$ of letters.  Given a partition, we represent the integers of the same partition block by the same letter, and   the first occurrence of each letter is in alphabetical order and vice versa. For example, the partition $\{\{1,3\}, \{2,4,5\}\}$ of $[5]$ corresponds to the word $ababb$. On the other hand, the word $aabccba$ represents the partition $\{\{1,2,7\},\{3,6\},\{4,5\}\}$  of $[7]$. A typical word will be denoted by $\boldsymbol{\omega}$ and its $i$-th letter as $\boldsymbol{\omega}[i]$. For example, for the word $\boldsymbol{\omega}=abbacac$, $\boldsymbol{\omega}[2]=b$ and the partition is $\{\{1,4,6\},\{2,3\},\{5,7\}\}$.

\noindent \textbf{Equivalence of Circuits}: For the sample covariance matrix, as there are two associated link functions $L1$and $L2$, the above holds with the following equivalence relation:
\begin{align*}
Lt(\pi_1(i-1),\pi_1(i))=Lt(\pi_1(j-1),\pi_1(j))
\Longleftrightarrow Lt(\pi_2(i-1),\pi_2(i))=Lt(\pi_2(j-1),\pi_2(j)), t=1,2.
\end{align*} 

%: Attached to any word $\boldsymbol {\omega}$, is an equivalence class of circuits $\Pi(\boldsymbol {\omega})$: 
%\begin{align*}
%\Pi(\boldsymbol {\omega}) &= \Big\{\pi: \boldsymbol {\omega}[i]=\boldsymbol {\omega}[j] \Leftrightarrow L(\pi(i-1),\pi(i))= L(\pi(j-1),\pi(j))\Big\}.
%\end{align*}
%This implies that for any word $\boldsymbol{\omega}$ of length $k$, the cardinality of $\pi(\boldsymbol{\omega})$ is
%\begin{align}\label{Pi(omega)}
%\big|\Pi(\boldsymbol {\omega})\big| = \big|\big\{ &\big(\pi(0), \pi(1),\ldots,\pi(k)\big): 1\leq \pi(i)\leq n \text{ for } i=0,1,\ldots,k,\ \pi(0)=\pi(k), \nonumber\\
% &\quad L(\pi(i-1),\pi(i))= L(\pi(j-1),\pi(j)) \text{ if and only if }\boldsymbol {\omega}[i]=\boldsymbol {\omega}[j]  \big\}\big|.
%\end{align}

\noindent \textbf{The class $\Pi(\boldsymbol {\omega})$}: First recall the notion of class $\Pi(\boldsymbol {\omega})$ from Section 3 of \citep{bose2021random}. Now suppose $\boldsymbol {\omega}$ is a word arising from the sample covariance matrix. Then for $\boldsymbol {\omega}$, $\boldsymbol {\omega}[i]=\boldsymbol {\omega}[j] \ \Leftrightarrow \xi_{\pi}(i)= \xi_{\pi}(j)\}$. This  implies 
\begin{align*}
Lt(\pi(i-1),\pi(i))& = Lt(\pi(j-1),\pi(j)) \ \ \ \text{ if } i  \ \text{ and } j \ \text{ are of same parity}, \ t=1,2 \\
Lt(\pi(i-1),\pi(i))& = Lt^{\prime}(\pi(j-1),\pi(j)) \ \ \ \text{ if } i  \ \text{ and } j \ \text{ are of different parity}, \ t,t^{\prime}\in \{1,2\}, t\neq t^{\prime}. 
\end{align*}
Therefore the class $\Pi_S(\boldsymbol {\omega})$ is given as follows:
\begin{align}\label{S-link}
\Pi_S(\boldsymbol {\omega}) & = \{\pi; \boldsymbol {\omega}[i]=\boldsymbol {\omega}[j] \ \Leftrightarrow \xi_{\pi}(i)= \xi_{\pi}(j)\} \nonumber \\
& = \Big\{\pi: \boldsymbol {\omega}[i]=\boldsymbol {\omega}[j] \Leftrightarrow (\pi(i-1),\pi(i))= (\pi(j-1),\pi(j))\ \ \text{ or } (\pi(i-1),\pi(i))= (\pi(j),\pi(j-1))\Big\}.
\end{align}

Recall from \eqref{genmoment-XXt} that the $k$th moment of the $S-$matrix involves the $2k$th moment of the entries of $X_p$. 

\vskip5pt
As we mentioned in Section \ref{introduction}, we will find a connection between the limit moments of the Wigner matrix and the covariance matrix. We will also find out that the link function of the Wigner matrix and the partitions that contribute to its limiting moments plays a crucial role in finding the LSD of the Sample covariance matrix. 
%Here, we recall the Wigner link function and the equivalence classes associated to it. (See Section 3 in \citep{bose2021random} for details).

We denote the link function of the Wigner matrix to be $L_{W}(i,j)=(\min(i,j),\max(i,j))$. For words with the Wigner link function $L_W$, the class $\Pi_W(\boldsymbol {\omega})$ is given as follows:
\begin{align}\label{Wignerlink}
\Pi_W(\boldsymbol {\omega}) &= \Big\{\pi: \boldsymbol {\omega}[i]=\boldsymbol {\omega}[j] \Leftrightarrow L_W(\pi(i-1),\pi(i))= L_W(\pi(j-1),\pi(j))\Big\} \nonumber \\
&= \Big\{\pi: \boldsymbol {\omega}[i]=\boldsymbol {\omega}[j]  \Leftrightarrow (\pi(i-1),\pi(i))= (\pi(j-1),\pi(j))
 \text{ or } (\pi(i-1),\pi(i))= (\pi(j),\pi(j-1))\Big\}.
\end{align}
Next, we make a key observation about the classes $\Pi_S(\boldsymbol {\omega})$ and $\Pi_W(\boldsymbol {\omega})$. 
\vskip5pt
\noindent \textbf{Observation 1}: Let $\tilde{\Pi}_W(\boldsymbol {\omega})$ be the possibly larger class of the circuits for the Wigner Link function with range $1 \leq \pi(i) \leq \max(p,n), 0 \leq i \leq 2k$. Then for a word $\boldsymbol{\omega}$, 
\begin{equation}\label{W-S}
\Pi_S(\boldsymbol {\omega})\subset \tilde{\Pi}_W(\boldsymbol {\omega}). 
\end{equation}

Next we recall the definition of \textit{generating and non-generating vertices} from \citep{bose2021random}.
\begin{definition} 
If $\pi$ is a circuit then any $\pi(i)$  will be called a \textit{vertex}. This vertex is \textit{generating} if $i=0$ or $\boldsymbol {\omega}[i]$ is the first occurrence of a letter in the  word $\boldsymbol {\omega}$ corresponding to $\pi$. All other vertices are \textit{non-generating}.
\end{definition} 

For example, for the word $abc$,  $\pi(0), \pi(1), \pi(2)$ and $\pi(3)=\pi(0)$  are generating.
For the word $aaa$, $\pi(0)$ and $\pi(1)$ are generating.
For the word $abcabc$ $\pi(0), \pi(1), \pi(2)$ and $\pi(3)$ are generating. It so happens that in this case due to the structure of the word, $\pi(3)=\pi(0)$. \\

\vskip5pt
 
\noindent \textbf{Even and odd generating vertices:} For the sample covariance matrix, $1 \leq \pi(2i)\leq p, \forall 1 \leq i \leq k; 1 \leq \pi(2i-1)\leq n, \forall 1 \leq i \leq k $. A generating vertex $\pi(i)$ is called even (odd) if $i$ is even (odd). Note that any word has at least one even and one odd generating vertices as $\pi(0)$ is an even generating vertex and $\pi(1)$ is an odd generating vertex. So for a matched word with $b(\leq k/2)$ distinct letters there can be $(r+1)$ even generating vertices where $0 \leq r \leq b-1$.

Observe that  
\begin{align}\label{S-Pi(omega)}
\big|\Pi_S(\boldsymbol {\omega})\big| = \big|\big\{ &\big(\pi(0), \pi(1),\ldots,\pi(2k)\big): 1\leq \pi(2i)\leq p, 1 \leq \pi(2i-1)\leq n \text{ for } i=0,1,\ldots,k,  \nonumber\\
 &\pi(0)=\pi(2k), \quad \xi_{\pi}(i)=\xi_{\pi}(j) \text{ if and only if }\boldsymbol {\omega}[i]=\boldsymbol {\omega}[j]  \big\}\big|.
\end{align}
Note that the circuits corresponding to a word $\boldsymbol {\omega}$ are completely determined by the generating vertices. 
The vertex $\pi(0)$ is always generating, and there is one generating vertex for each new letter in $\boldsymbol{\omega}$. So, if $\boldsymbol {\omega}$ has $b$ distinct letters then the number of generating vertices is $(b+1)$. As seen in the example above, the numerical value of some of these generating vertices may be identical, depending on the nature of the word. In any case, as $p/n\rightarrow y>0$,
\begin{equation}\label{cardiality of word}
| \Pi_S(\boldsymbol {\omega})| = \mathcal{O}(n^{b+1})\ \ \text{whenever} \ \omega \text{ has } \\   b \ \ \text{distinct letters}.
\end{equation} 

We shall see later that the existence of 
\begin{align}\label{word limit}
\lim_{n \rightarrow \infty} \frac{| \Pi_S(\boldsymbol {\omega})|}{n^{b+1}} 
\end{align}
is tied very intimately to the LSD of $S$. We shall look into it in the next section.

\section{Proofs}\label{proofs}
First, we recall a fact for words with the Wigner link function $L_W$.  
\begin{lemma}[Lemma 3.1 and 3.3 in \citep{bose2021random}]\label{W-words}
Consider the Wigner link function. For any word $\boldsymbol{\omega}$  with $b$ distinct letters  
\begin{equation}
\lim_{n \rightarrow \infty}\frac{| \Pi_W(\boldsymbol {\omega})|}{n^{b+1}}= \begin{cases}
1, & \ \  \omega\in SS_b(2k)\\
0, & \ \  \omega\in \mathcal{P}(2k) \setminus SS_b(2k).
\end{cases}
\end{equation} 
\end{lemma}

In the next lemma we explore the limit in \eqref{word limit}.
\begin{lemma}\label{S-words}
Let $\boldsymbol{\omega}$ be a word with $b$ distinct letters and $(r+1) \ (0 \leq r \leq b-1)$ even generating vertices. Then
\begin{equation}\label{limit S-words}
\lim_{n \rightarrow \infty}\frac{| \Pi_S(\boldsymbol {\omega})|}{n^{b+1}}= \begin{cases}
y^{r+1}, & \ \  \omega\in SS_b(2k)\\
0, & \ \  \omega \notin  SS_b(2k).
\end{cases}
\end{equation}
\end{lemma}

\begin{proof}
 First suppose $\boldsymbol {\omega} \in \mathcal{P}(2k) \setminus SS_b(2k)$. Then from \eqref{W-S} and Lemma \ref{W-words}, is easy to see that $$\lim_{n \rightarrow \infty}\frac{| \Pi_S(\boldsymbol {\omega})|}{n^{b+1}}=0.$$ \\

\noindent From the following claim it it immediately follows that $\lim_{n \rightarrow \infty}\frac{| \Pi_S(\boldsymbol {\omega})|}{n^{b+1}}= y^{r+1},  \ \ \text{ if } \boldsymbol{\omega} \in SS_b(2k)$.\\

\noindent \textbf{Claim}: suppose $\boldsymbol {\omega}\in SS_b(2k)$ with $(r+1)$ even generating vertices. $\big|\Pi_S(\boldsymbol {\omega}) \big| = p^{r+1}n^{b-r}$.\\

\vskip3pt

\noindent \textbf{Proof of the claim}: We use induction on $b$, the number of distinct letters, to prove the claim.

When $b=1$, then $r=0$ and $\boldsymbol {\omega}$ must be a string of $a$ of length $2k$. Therefore $\pi(0)$ and $\pi(1)$ are the generating vertices and both can be chosen freely. Thus, $\big|\Pi_S(\boldsymbol {\omega}) \big|=pn$.

 Assume that for any word $\omega$ with $(b-1)$ distinct letters and $(r+1),\ (0 \leq r\leq b-2)$ even generating vertices, $\big|\Pi_S(\boldsymbol {\omega}) \big|=p^{r+1}n^{b-1-r}$. 

Now it is enough to prove that if $\boldsymbol {\omega}$ has $b$ distinct letters with $(r+1),\ (0 \leq r \leq b-1)$ even generating vertices, then $\big|\Pi_S(\boldsymbol {\omega}) \big|=p^{r+1}n^{b-r}$.

First let $0 \leq r \leq b-2$. 

Now suppose the last distinct letter of $\boldsymbol {\omega}$, say, $z$ appears for the first time at the $i-$th position, that is at $(\pi(i-1),\pi(i))$ or $(\pi(i),\pi(i-1))$ (depending on whether $i$ is odd or even). As $\boldsymbol {\omega}\in SS_b(2k)$, $z$ appears in pure even blocks. Let the length of the first pure block of $z$ be $m$ ($m$ even). Then we have the following two cases:\\
\vskip3pt

\noindent Case 1: $i$ is odd. Then we have
\begin{align}\label{even z}
& \pi(i-1)= \pi(i+1)= \cdots = \pi(i+m-1), \nonumber \\
& \pi(i)= \pi(i+2)= \cdots = \pi(i+m-2).
\end{align}
Similar identities can be shown for all other pure blocks of $z$. Hence $\pi(i)$ can be chosen freely with $1 \leq \pi(i) \leq n$ as it does not appear elsewhere in $\boldsymbol {\omega}$ other than the letter $z$. Now dropping all $z$s from $\boldsymbol{\omega}$, we get a word $\boldsymbol {\omega}^{\prime}$ with $(b-1)$ distinct letters and $(r+1)$ even generating vertices. Noting the structure of special symmetric words, $\boldsymbol {\omega}^{\prime}$ is also a special symmetric word with $(b-1)$ distinct letters. Therefore, by induction hypothesis, $\big|\Pi_S(\boldsymbol {\omega}^{\prime}) \big|=p^{r+1}n^{b-(r+1)}$. Now as $\pi(i)$ is another odd vertex that can be chosen freely, we have $\big|\Pi_S(\boldsymbol {\omega}) \big|=p^{r+1}n^{b-(r+1)}n= p^{r+1}n^{b-r}$.\\
\vskip3pt

\noindent Case 2: $i$ is even. Then we have
\begin{align*}
& \pi(i-1)= \pi(i+1)= \cdots = \pi(i+m-1)\\
& \pi(i)= \pi(i+2)= \cdots = \pi(i+m-2).
\end{align*}
  As in case 1, the generating vertex $\pi(i)$ can be chosen freely with $1 \leq \pi(i) \leq p$. Now dropping all $z$s from $\boldsymbol{\omega}$ as before, we get a word $\boldsymbol {\omega}^{\prime}$ with $(b-1)$ distinct letters and $r$ even generating vertices. Also by the structure of special symmetric words, $\boldsymbol {\omega}^{\prime}$ is a special symmetric word with $(b-1)$ distinct letters. Therefore, by induction hypothesis, $\big|\Pi_S(\boldsymbol {\omega}^{\prime}) \big|=p^{r}n^{b-r}$. Now as $\pi(i)$ is another even vertex that can be chosen freely, we have $\big|\Pi_S(\boldsymbol {\omega}) \big|=p^{r}p n^{b-r}= p^{r+1}n^{b-r}$. 
  
  \vskip3pt
  
\noindent Now let $r=b-1$. Then there are $r+1=b$ even generating vertices (one of them being $\pi(0)$) and $b$ distinct letters in $\boldsymbol{\omega}$. Therefore all letters except the first appear for the first time at even positions in $\boldsymbol{\omega}$. So, if $z$ is the last distinct letter of $\boldsymbol{\omega}$, then $z$ appears for the first time at $(\pi(i-1),\pi(i))$ where $i$ is even. Thus we have \eqref{even z} similarly as case 1 above. Hence $\pi(i)$ can be chosen freely with $1 \leq \pi(i) \leq n$. Now dropping all $z$s as before from $\boldsymbol{\omega}$, we get a word $\boldsymbol {\omega}^{\prime}$ with $(b-1)$ distinct letters and $b-2$ even generating vertices. Clearly, $\boldsymbol {\omega}^{\prime}$ is also a special symmetric word with $(b-1)$ distinct letters. Therefore, by induction hypothesis, $\big|\Pi_S(\boldsymbol {\omega}^{\prime}) \big|=p^{b-1}n^{b-(b-1)}$. Now as $\pi(i)$ is another even vertex that can be chosen freely, we have $\big|\Pi_S(\boldsymbol {\omega}) \big|=p^{b-1}n p= p^{b} n= p^{r+1}n^{b-r}$, $r=b-1$. 
 
 Hence the claim is proved. This completes the proof of the lemma. 
\end{proof}
Next, we state the following elementary result that shall be useful in the proof of Theorem \ref{res:XXt}. 
See Section 1.2 of \citep{bose2018patterned} for a proof.
\begin{lemma}\label{lem:genmoment}
Suppose $A_n$ is any sequence of symmetric random matrices such that the following conditions hold:
\begin{enumerate}
\item[(i)] For every $k\geq 1$, $\frac{1}{n}\mathbb{E}[\Tr (A_n)^k] \rightarrow \alpha_k$ as $n \rightarrow \infty$.
\item[(ii)] $\displaystyle \sum_{n=1}^{\infty}\frac{1}{n^4}\mathbb{E}[\Tr(A_n^k) \ - \ \mathbb{E}(\Tr(A_n^k))]^4  < \infty$ for every $k \geq 1$.
\item[(iii)] The sequence $\{\alpha_k\}$ is the moment sequence of a unique probability measure $\mu$.
\end{enumerate}
Then $\mu_{A_n}$ converges to $\mu$
weakly almost surely.
\end{lemma}

Condition (ii) of Lemma \ref{lem:genmoment} will be referred to as the \textit{fourth moment condition}. The following lemma for Wigner matrices will be important in proving the fourth moment condition for the $S-$matrix.
\begin{lemma}[Lemma 4.2 in \citep{bose2021random}]\label{lem:moment}
For the Wigner link function , let 
\begin{align*}
Q_{k,4}^b = | \{& (\pi_1,\pi_2,\pi_3,\pi_4): \ell(\pi_i)=2k; \pi_i, 1 \leq i \leq 4 \ \text{jointly- and cross-matched with }\\
 & b \text{ distinct edges or } b \text{ distinct letters across all } (\pi_i)_{1\leq i \leq 4}\}|.
\end{align*}
 Then there exists a constant C, such that,
\begin{equation}
Q_{k,4}^b \leq  C \ n^{b +2}\ .
\end{equation} 
\end{lemma}
The arguments in proof of Lemma 5.4 can be used to prove the same for the $S-$ link function as $1 \leq \pi(2i)\leq p$ and $1 \leq \pi(2i-1)\leq n$ and $p$ and $n$ are comparable for large $n$.\\

The next lemma is a well-known result that is useful in the proof of the theorem. For a proof of the lemma see Corollary A.42 in \citep{bai2010spectral}.
\begin{lemma}\label{lem:metric}
Suppose $A$ and $B$ are real $p \times n$ matrices and $F^{S_A}$ and $F^{S_B}$ denote the ESDs of $AA^T$ and $BB^T$ respectively. Then the \textit{Levy distance}, $L$ between the distributions $F^{S_A}$ and $F^{S_B}$ satisfy the following inequality: 
\begin{equation}\label{levy}
L^4(F^A,F^B)\leq \frac{2}{p^2}(\Tr(AA^T+BB^T))(\Tr[(A-B)(A-B)^T]).
\end{equation}
\end{lemma}	
\begin{remark}\label{convergence}
It is well-known that for a sequence of probability measure $\{\mu_n\}$ and a probability measure $\mu$, $L^4(\mu_n,\mu) \rightarrow 0$ as $n \rightarrow \infty$, implies $\mu_n$ converges weakly to $\mu$.
\end{remark}

\subsection{Proof of Theorem \ref{res:XXt} and Corollary \ref{p=n}}

\begin{proof}[\textit{Proof of Theorem \ref{res:XXt} }]
 To prove (a), we make use of Lemma \ref{lem:genmoment} and use the notion of words and circuits in order to calculate the moments. 

We break the proof  into a few steps.

\vskip5pt
\noindent\textbf{Step 1:} We reduce the general case to the case where all the entries of $Z_p$ have mean 0. For this, consider the matrix $\widetilde{Z}_p$ whose entries are $(y_{ij}-\mathbb{E}y_{ij})$ and thus have mean 0. Now
\begin{align}\label{meanzero-noiid}
n\ \mathbb{E}[(y_{ij}- \mathbb{E} y_{ij})^{2k}]= n\ \mathbb{E}[y_{ij}^{2k}] + n \displaystyle \sum_{t=0}^{2k-1} {{2k} \choose {t}}\mathbb{E}[y_{ij}^{t}]\ (\mathbb{E}y_{ij})^{2k-t}.  
\end{align}
The first term of the r.h.s. equals $g_{2k}(i/p,j/n)$ by \eqref{gkeven}. The second term is tackled as follows:
\begin{align*}
\text{For } t\neq {2k-1}, \ \ n \ \mathbb{E} [y_{ij}^{t}]\ (\mathbb{E}y_{ij})^{2k-t}& = (n^{\frac{1}{2k-t}}\ \mathbb{E}y_{ij})^{2k-t} \ \mathbb{E}[y_{ij}^{t}]\\
&\overset{n \rightarrow \infty}{\longrightarrow} 0, \ \ \ \text{ by condition } \eqref{gkodd}.  
\end{align*}
\begin{align*}
\text{For } t={2k-1}, \ \ n \ \mathbb{E} [y_{ij}^{2k-1}]\ \mathbb{E}y_{ij}& = (\sqrt{n} \ \mathbb{E} [y_{ij}^{2k-1}]) \ (\sqrt{n}\  \mathbb{E}y_{ij})\\
&\overset{n \rightarrow \infty}{\longrightarrow} 0, \ \ \ \text{ by condition } \eqref{gkodd}.
\end{align*}
Hence from \eqref{meanzero-noiid}, we see condition \eqref{gkeven} is true for the matrix $\widetilde{Z}_p$. Similarly we can show that \eqref{gkodd} is true for $\widetilde{Z}_p$. Hence, Assumption A holds for the matrix $\widetilde{Z}_p$.

Now from Lemma \ref{lem:metric}, 
\begin{align*}
L^4\big(F^{S_{Z_p}}, F^{S_{\widetilde{Z}_p}}\big) & \leq  \frac{2}{p^2}(\Tr(Z_pZ_p^T+\widetilde{Z}_p\widetilde{Z}_p^T))(\Tr[(Z_p-\widetilde{Z}_p)(Z_p-\widetilde{Z}_p)^T]) \\
&\leq \frac{2}{p^2}\bigg(\displaystyle \sum_{i,j}\big( 2y_{ij}^2+ (\mathbb{E}y_{ij})^2- 2y_{ij}\mathbb{E}y_{ij}\big)\bigg)\bigg(\sum_{i,j}(\mathbb{E}y_{ij})^2\bigg)
\end{align*}
The second factor of the rhs in the above inequality is bounded by $$n (\sup_{i,j}\mathbb{E}y_{ij})^2=(\sup_{i,j}\sqrt{n}\mathbb{E}y_{ij})^2\rightarrow 0  \ \ \text{ as } n \rightarrow \infty \ \ \text{ by } \eqref{gkodd}. $$ Now it can be seen applying Borel-Cantelli lemma that $\displaystyle \sum_{i,j} (y_{ij}^2-\mathbb{E}[y_{ij}^2] \rightarrow 0$ almost surely as $p \rightarrow \infty$ (proof is given in Section \ref{appendix}). Also $\mathbb{E}\big[\frac{1}{p}\sum_{ij}y_{ij}^2\big] \rightarrow \int_{[0,1]^2} g_2(x,y)\ dx\ dy$. Hence, $$\mathbb{P}\big[\{\omega; \limsup_p\frac{1}{p}\displaystyle \sum_{i,j}y_{ij}^2(\omega)= \infty\}\big]=0.$$ Therefore the first term of the rhs in the inequality also tends to zero almost surely.
From this observation and Remark \ref{convergence},  the LSD of $S_{Z_p}$ and $S_{\widetilde{Z}_p}$ are same. Thus we can assume that the entries of $Z_p$ have mean 0.

To prove the first part of the theorem, we shall use Lemma \ref{lem:genmoment}. We will verify the conditions (i), (ii) and (iii) of the lemma using Assumption A and a few other observations made earlier.  
\vskip5pt
\noindent\textbf{Step 2:} We verify condition (ii) of Lemma \ref{lem:genmoment} for $S_{Z_p}$ in this step.
 Observe that 
 \begin{align}\label{fourthmoment-XXt}
 \frac{1}{p^{4}} \mathbb{E}[\Tr(Z_pZ_p^T)^k \ - \ \mathbb{E}(\Tr(Z_pZ_p^T)^k)]^4\ = \frac{1}{p^4} \displaystyle \sum_{\pi_1,\pi_2,\pi_3,\pi_4} \mathbb{E}[\displaystyle \Pi_{i=1}^4 (Y_{\pi_i}\ - \ \mathbb{E}Y_{\pi_i})].
 \end{align}
 If $(\pi_1,\pi_2,\pi_3,\pi_4)$ are not jointly-matched, then one of the circuits has a letter that does not appear elsewhere. 
	Hence by independence and mean zero assumption, $\mathbb{E}[\displaystyle \Pi_{i=1}^4 (Y_{\pi_i}\ - \ \mathbb{E}Y_{\pi_i})]=0$.
  If $(\pi_1,\pi_2,\pi_3,\pi_4)$ are not cross-matched, then one of the circuits say $\pi_j$ is only self-matched. Then we have $\mathbb{E}[Y_{\pi_j}\ - \ \mathbb{E}Y_{\pi_j}]=0$. So again we have  $\mathbb{E}[\displaystyle \Pi_{i=1}^4 (Y_{\pi_i}\ - \ \mathbb{E}Y_{\pi_i})]=0$.

So we consider only circuits $(\pi_1,\pi_2,\pi_3,\pi_4)$ that are jointly- and cross-matched. Here each circuit is of length $2k$, so the total number of edges($L-$ values) is $8k$. As the circuits are at least pair-matched, the number of distinct edges is at most $4k$. 

Suppose $\pi_i$ has $k_i$ distinct letters, $1\leq i \leq 4$ with $k_1+k_2+k_3+k_4=b$. Suppose the $j$th distinct letter appears $s_j$ times across $\pi_1,\pi_2,\pi_3,\pi_4$ and first at the $i_j-$th position.
Let $b_1$ and $b_2$ ($b_1+b_2=b$) be respectively the number of even and odd $s_i$'s, denoted by $s_{i_1},s_{i_2},\ldots,s_{i_{b_1}}$and $s_{i_{b_1+1}},s_{i_{b_1+2}},\ldots,s_{i_{b_2}}$.  Each term can then be written as 
\begin{align*}
 \frac{1}{p^4}  \displaystyle \sum_{b=1}^{4k} p^{-{b_1}} p^{-(b_2-\frac{1}{2})}
 \prod_{j=1}^{b_1} \  g_{s_{i_j},n}(\pi(i_j-1)/p,\pi(i_j)/n) \ \prod_{m=b_1+1}^{b_1+b_2} n^{\frac{b_2-(1-1/2)}{b_2}} \mathbb{E}[y_{\pi(i_{m}-1)\pi(i_{m})}^{s_{i_m}}]. 
 \end{align*}

We note that $g_{s_{i_j},n}   \rightarrow  g_{s_{i_j}}$  for all $1\leq j \leq b_1$. Therefore, the sequence $\|g_{s_{i_j},n}\|$ is bounded by  a constant $M_{j}$. Also as $\frac{b_2-(1-1/2)}{b_2}<1$, by \eqref{gkodd}, we have $n^{\frac{b_2-(1-1/2)}{b_2}} \mathbb{E}[y_{\pi(i_{m}-1)\pi(i_{m})}^{s_{i_m}}]  $ is bounded by $1$ for $n$ large when $b_1+1\leq m \leq b_1+b_2$. Let
$$M^{\prime} = \underset{b_1+b_2=b}{\max} \{M_{t},1:  1 \leq t \leq b_1 \}\ \mbox{ and }\ M_0^{\prime}=  \max \{{M^{\prime}}^b:  1 \leq b \leq 2k \}. 
$$
Now using \eqref{W-S} and Lemma \ref{lem:moment}, we can say that the number of such circuits having $b$ distinct letters ($b=1,\ldots, k$) is bounded by $C_1n^{b+2}$ for some constant $C_1>0$. Therefore we have with $C_2=y^{b+2}C_1$,
 \begin{align*}
\frac{1}{p^{4}} \mathbb{E}[\Tr(Z_pZ_p^T)^k \ - \ \mathbb{E}(\Tr(Z_pZ_p^T)^k)]^4\
& \leq C_2 M_0^{\prime} \displaystyle \sum_{b=1}^{4k} \frac{1}{p^{b+3\frac{1}{2}}} p^{b+2} 
= \mathcal{O}(p^{-\frac{3}{2}}).
\end{align*}

This completes  the  proof of Step 2.\\

\vskip3pt

\noindent  By Lemma \ref{lem:genmoment} and the previous step, it is now enough to show  that for every $k\geq 1$, \\
   $\displaystyle \lim_{n\to\infty}\frac{1}{n}\mathbb{E}[\Tr(Z_n)^{k}]$ exists and is given by $\beta_k( \mu^{\prime})$ for each $k \geq 1$.\\

\vskip5pt

\noindent\textbf{Step 3:} We verify condition, (i) of Lemma \ref{lem:genmoment} for $Z_p$ in this step. 

First note that, we can write \eqref{genmoment-XXt} as \begin{align}\label{moment-XXt1}
\lim_{n \rightarrow \infty}\frac{1}{p}\mathbb{E}[\Tr(Z_pZ_p^*)^k]
 &=   \displaystyle\lim_{n \rightarrow \infty} \displaystyle \sum_{b=1}^k  \Big[\frac{1}{p} \sum_{\omega \in SS_b(2k)} \sum_{\pi \in \Pi(\boldsymbol{\omega})}\ \mathbb{E}(Y_{\pi}) + \frac{1}{n} \sum_{\underset{\boldsymbol {\omega} \text{ with b letters}}{\omega \notin SS(2k)}}  \sum_{\pi \in \Pi(\boldsymbol{\omega})}\ \mathbb{E}(Y_{\pi})\Big] .\nonumber \\
& = T_1+T_2.
\end{align}
Suppose that $\boldsymbol {\omega}$ has $b$ distinct letters and let $\pi \in \Pi_S(\boldsymbol {\omega})$. Suppose the first appearance of the letters of $\boldsymbol {\omega}$ are at the $i_1,i_2, \ldots,i_b$ positions. So the $j$th new letter appears at the $(\pi(i_j-1),\pi(i_j))-$th position for the first time. Let $D$ denote the set of all distinct generating vertices. Thus $|D|\leq (b+1)$.

If $\boldsymbol {\omega}$ is a word with $b$ distinct letters that does not belong to $SS(2k)$, then the contribution of  $\boldsymbol {\omega}$ to $T_2$ in the above sum is $0$. Indeed, from Lemma \ref{S-words}, it follows that for $\boldsymbol {\omega} \notin  SS_b(2k)$, $|D|\leq b$ and hence $\boldsymbol {\omega}$ has no contribution. Hence $T_2$ has no contribution in \eqref{moment-XXt1}.

Now suppose $\boldsymbol {\omega} \in SS_b(2k)$ with $(r+1)$ even generating vertices. Clearly, by Lemma \ref{S-words}, $\boldsymbol {\omega}$ has $(b+1)$ distinct generating vertices. For each $j \in \{1,2,\ldots , b\}$ denote $(\pi(i_j-1),\pi(i_j))$ as $(t_j,l_j)$. Clearly $t_1=\pi(0)$ and $l_1=\pi(1)$. It is easy to see that each distinct $(t_j,l_j)$ corresponds to each distinct letter in $\boldsymbol {\omega}$. Suppose the $j$th new letter appears $s_j$ times in $\boldsymbol {\omega}$. Clearly all the $s_j$ are even. So the total contribution of this $\boldsymbol {\omega}$ to $T_1$ in \eqref{moment-XXt1} is: 
\begin{align}\label{finitesum-XXt}
\frac{1}{pn^b} \displaystyle \sum_{S} \prod_{j=1}^b g_{s_j,n}(t_j/p,l_j/n)
\end{align} 
Recall that there are $(r+1)$ even generating vertices in $D$ with range between $1$ and $p$ and $(b-r)$ vertices (odd generating) with range between $1$ and $n$. So as $n \rightarrow \infty$, \eqref{finitesum-XXt} converges to
\begin{equation}\label{ss2k-limit}
y^r \int_{[0,1]^{b+1}} \displaystyle \prod_{j=1}^b g_{k_j}(x_{t_j},x_{l_j})\  dx_S
\end{equation}
where $dx_S= \prod_{i \in S} dx_{i}$ denotes the $(b+1)-$dimensional Lebesgue measure on $[0,1]^{b+1}$.

Hence we obtain
\begin{align}\label{moment-XXt}
\displaystyle \lim_{p \rightarrow \infty} \frac{1}{p}\mathbb{E}[\Tr S_p^k]= \sum_{b=1}^k\sum_{r=0}^{b-1} \sum_{\underset{\underset{\text{ even generating vertices}}{\text{having }(r+1)}}{\pi \in SS_b(2k)}} y^r \int_{[0,1]^{b+1}} \displaystyle \prod_{j=1}^b g_{k_j}(x_{t_j},x_{l_j})\  dx_S.
 \end{align}
 
This completes the verification of the first moment condition.
 \vskip5pt
\noindent \textbf{Step 4:} We prove the uniqueness of the measure in this step. We have obtained 
\begin{align*}
\gamma_{2k}= &\lim_{p \rightarrow \infty}\frac{1}{p}\mathbb{E}[\Tr(S_p)^{2k}]\leq \sum_{b=1}^k\sum_{r=0}^{b-1} \sum_{\underset{\underset{\text{ even generating vertices}}{\text{having }(r+1)}}{\sigma \in SS_b(2k)}} y^r M_{\sigma}
\end{align*}
Let $c=\max(y,1)$. Then 
\begin{align*}
\gamma_{2k} \leq \displaystyle \sum_{\sigma \in SS(2k)} c^{k}M_{\sigma}
 \leq \displaystyle \sum_{\sigma \in \mathcal{P}(2k)} c^k M_{\sigma} 
 =c^k \alpha_{2k}.
\end{align*}
 As $\{\alpha_{2k}\}$ satisfies Carleman's condition, $\{\gamma_{2k}\}$ also does so. Now using Lemma \ref{lem:genmoment}, we see that there exists a measure $ \mu$ with  moments $\{\beta_k(\mu)=\gamma_{2k}\}_{k \geq 1}$ such that $\mu_{S_{Z_p}}$ converges weakly almost surely to $ \mu$.

 This completes the proof of part (a).
\vskip5pt
 
\noindent\textbf{Step 5:} In this final step we prove part (b) of the theorem. Observe that from Lemma \ref{lem:metric}, we have 
\begin{align}\label{levy-XXt}
L^4(F^{S_p},F^{S_{Z_p}}) &\leq \frac{2}{p^2}(\Tr(X_pX_p^T+Z_pZ_p^T))(\Tr[(X_p-Z_p)(X_p-Z_p)^T]) \nonumber\\
& = \frac{2}{p} \bigg(\displaystyle2 \sum_{i,j}y_{ij}^2+\sum_{ij}x_{ij}^2\boldsymbol{1}_{[|x_{ij}|>t_n]} \bigg)\bigg(\frac{1}{p}\sum_{ij}x_{ij}^2\boldsymbol{1}_{[|x_{ij}|>t_n]}\bigg)
\end{align}
The second factor in the above equation tends to zero almost surely (or in probability) as $n \rightarrow \infty$ due to the assumptions $\frac{1}{n} \sum_{i,j} \ x_{ij}^2\boldsymbol {1}_{\{|x_{ij}| > t_n\}} \rightarrow 0, \ \mbox{almost surely (or in probability)}$ and $p/n\rightarrow y \in (0,\infty)$. Now as $\displaystyle\frac{1}{p} \sum_{i,j} (y_{ij}^2-\mathbb{E}[y_{ij}^2]) \rightarrow 0 \ \ \ \text{ almost surely as }$ (see Section \ref{appendix}) and $\mathbb{E}\big[\frac{1}{p}\sum_{ij}y_{ij}^2\big] \rightarrow \int_{[0,1]^2} g_2(x,y)\ dx\ dy$, and hence is finite, $\frac{1}{p}\sum_{i,j}y_{ij}^2$ is bounded almost surely. %$\mathbb{P}\big[\{\omega; \limsup_p\frac{1}{p}\displaystyle \sum_{i,j}y_{ij}^2(\omega)= \infty\}\big]=0$.
 Therefore the first factor in \eqref{levy-XXt} also tends to zero either almost surely or in probability.

 Therefore from the discussion in the above paragraph and Reamrk \ref{convergence}, we get that the ESD of $X_pX_p^T$ converges weakly to the probability measure $\mu$ almost surely (or in probability).
This completes the proof of the theorem. 
\end{proof}
\begin{proof}[\textbf{Proof of Corollary \ref{p=n}}]
In the case $p=n$, if $\{g_{2k,n}\}$ are symmetric functions, then the assumption on the entries of $X_p$ are no different from that on the entries of $W_n$ in Theorem 2.1 of \citep{bose2021random}. Now from \eqref{moment-XXt} and equation (4.11) in \citep{bose2021random} we see that $\mathbb{E}[X^k]=\mathbb{E}[Y^{2k}], k \geq 1$. Hoever observe that even if $\{g_{2k,n}\}$ are not symmetric for every $n$, but $\{g_{2k}\}$ are symmetric functions $\mathbb{E}[X^k]=\mathbb{E}[Y^{2k}], k \geq 1$ still holds. Therefore, by the uniqueness of the probability distribution via moments, we have $X \overset{\mathcal{D}}{=} Y^2$.
\end{proof}

\begin{proof}[Proof of Remark \ref{unbounded support}]
Consider $k=mt$ for some $t \geq 1$. Then from \eqref{moment-XXt}, we have that 
\begin{align}\label{moment-unbounded support1}
\beta_{k}(\mu)= \sum_{b=1}^k\sum_{r=0}^{b-1} \sum_{\underset{\underset{\text{ even generating vertices}}{\text{having }(r+1)}}{\pi \in SS_b(2k)}} y^r \int_{[0,1]^{b+1}} \displaystyle \prod_{j=1}^b g_{k_j}(x_{t_j},x_{l_j})\  dx_S.
\end{align}
Recall that $\pi$ in the above expression could be described as a word in $SS_{b}(2k)$ having $(r+1)$ even generating vertices. Now consider all words $\omega$ with $t$ distinct letters such that each letter appears $2m$ times and also in pure even blocks in $\omega$. Clearly $\omega\in SS_t(2k)$ with only one even generating vertex $\pi(0)$. Therefore as $n \rightarrow \infty$, the contribution of such $\omega$ in the limiting moment is as follows (see \eqref{ss2k-limit}):
\begin{equation}\label{ss2k-unbounded}
\int_{[0,1]^{t+1}} g_{2m}(x_0,x_1)g_{2m}(x_0,x_2)\cdots g_{2m}(x_0,x_t) \ dx_0dx_1\cdots dx_t= \int_{[0,1]} \big(f_{2m}(x_0)\big)^t \ dx_0 .
\end{equation}
Next, observe that the number of words that fits the description given in the above paragraph (words such as $\omega$ above) is
\begin{equation}\label{number of such words}
\frac{1}{t!} {{mt} \choose m}{{mt-t} \choose m}\cdots {m \choose m}= \frac{1}{t!}\frac{(mt)!}{(m!)^t}. 
\end{equation}
Now as the integrand in the \eqref{moment-unbounded support1} is non-negative for all special symmetric words, using \eqref{ss2k-unbounded} and \eqref{number of such words}, we have
\begin{align*} 
\beta_{k}(\mu)& > \frac{1}{t!}\frac{(mt)!}{(m!)^t} \int_{[0,1]} \big(f_{2m}(x_0)\big)^t \ dx_0 . \nonumber\\
& = \frac{(mt)!}{t!} \int_{[0,1]}\bigg(\frac{f_{2m}(x_0)}{m!}\bigg)^t \ dx_0 > c \frac{(mt)!}{t!} , k=mt.
\end{align*}
Therefor for $t$ sufficiently large (with $k=mt$), 
\begin{equation*}
(\beta_k(\mu))^{1/k}> K \ t^{\eta} \ \ \text{ for some constant } K>0 \  \text{ and } \eta>0.
\end{equation*}
Therefore $(\beta_k(\mu))^{1/k} \rightarrow \infty$ as $k=mt \rightarrow \infty$. Hence $\mu$ has unbounded support.
\end{proof}

\noindent \textbf{Moments of the variance profile matrices}: Now we give a description of the limiting moments of the S-matrices with variance profile. Observe that, from Step 3 in the proof of Theorem \ref{res:XXt}, for each word in $SS_{b}(2k)$ with $(r+1)$ even genrating vertices and each distinct letter appearing $s_1,s_2,\ldots,s_b$ times, its contribution to the limiting moments is (see \eqref{ss2k-limit})
$$y^r\int_{[0,1]^{b+1}} \prod_{j=1}^b \sigma^{s_j}(x_{t_j},x_{l_j}) \ \prod_{i\in S} dx_i \prod_{j=1}^b C_{s_j},$$
where $(t_j,l_j)$ denotes the position of first appearance of the $j$th distinct letter in the word.

Hence the $k$th moment of $\nu$ is
\begin{align*}
\beta_{k}(\nu)= \sum_{b=1}^k\sum_{r=0}^{b-1} \sum_{\underset{\underset{\text{ even generating vertices}}{\text{having }(r+1)}}{\pi \in SS_b(2k)}} y^r \int_{[0,1]^{b+1}} \displaystyle \prod_{j=1}^b \sigma^{s_j}(x_{t_j},x_{l_j}) dx_i \prod_{j=1}^b C_{s_j}.
\end{align*}

\section{Hypergraphs, Noiry-words and $SS(2k)$}\label{connection-ss2k}
In Section \ref{triangular iid}, we discussed about entries that have triangular iid distribution and we found how we can conclude their convergence using Thoerem \ref{res:XXt}. We also described the limiting moments via $SS(2k)$ partitions. Here, we show how the moments that we have obtained are the same with those obtained in \citep{Benaych-Georges2012} and \citep{noiry2018spectral}. 

\begin{definition}
Let $G$ be a graph with vertex set $V$ and let $\pi$ be a partition of $V$ and $\tau$ be a partition of the edge set. Then $H(\pi,\tau)$  is defined to be the hypergraph with vertex set as $G_{\pi}$ (i.e. $\pi$) and edges $\{E_W; W\in \tau\}$, where each edge $E_W$ is the set of blocks $J \in \pi$ such that at least one edge of $G_{\pi}$ starting or ending at $J$ belongs to $W$.
Further if no two of the edges can have more than one common vertex, then $H(\pi,\tau)$ is said to be a \textit{hypergraph with no cycle}.
\end{definition}
In \citep{Benaych-Georges2012}, equation (22) describes the limiting moments as a sum on \textit{Hypergraphs with no cycle}. For details on Hypergraphs, see Sections 5.3 and 12.3.2 in \citep{Benaych-Georges2012} and \citep{berge1984hypergraphs}. 
\begin{lemma}\label{hypergraph-ss2k}
For every word $\boldsymbol{\omega} \in SS_b(2k)$, there exists partitions $\pi,\tau \in \mathcal{P}(k)$ such that there is a unique hypergraph $H(\pi,\tau)$ which has no cycle with $|\pi|+|\tau|=b+1$. The converse is also true.
\end{lemma}
\begin{proof}
Suppose $\boldsymbol{\omega} \in SS_b(2k)$ with $(r+1)$ even generating vertices. Then using arguments similar to Lemma \ref{S-words} we find that there are $(r+1)$ even generating vertices and $(b-r)$ odd generating vertices. Suppose the $b$ distinct letters appear in the $i_1,i_2,\ldots,i_b$ positions for the first time in $\boldsymbol{\omega}$, i.e., the first appearance of the $j$th letter is at $(\pi(i_j-1),\pi(i_j))$ or $(\pi(i_j),\pi(i_j-1))$ (depending on whether $i_j$ is odd or even). Suppose the even generating vertices of $\boldsymbol{\omega}$ are $\pi(i_{t_0}),\pi(i_{t_1}),\ldots,\pi(i_{t)r})$ where $\pi(i_{t_0})=\pi(0)$ and the odd generating vertices are $\pi(i_{m_1}),\pi(i_{m_2}),\ldots ,\pi(i_{m_{b-r}})$ where $\pi(i_{m_1})=\pi(1)$. Let $V_j= \{\pi(2i): \pi(2i)=\pi(i_{t_j}), 1 \leq i \leq k\}$, $0 \leq j \leq r$ and $W_j= \{\pi(2i-1): \pi(2i-1)=\pi(i_{m_j}), 1 \leq i \leq k\}$, $1 \leq j \leq (b-r)$. Clearly, $\sigma=\{V_j; 0 \leq j \leq r\}$ and $\tau=\{W_j; 1 \leq j \leq (b-r)\}$ are two partitions of $\{1,2,\ldots,k\}$. Therefore, we can construct a hypergraph $H(\sigma,\tau)$ where $\sigma$ is the vertex set and $\{E_W;W\in \tau\}$ is the edge set (see \ref{hypergraph-ss2k}). Now suppose that this hypergraph has a cycle. That means by construction, there exists $a,b (a\neq b)\in \{1,2,\ldots, (b-r)\}$ and $q,l (q\neq l)\in \{0,1, \ldots,r\}$ such that $V_q, V_l \in W_a \cap W_b$. That is, there are edges $(\pi(k_1-1),\pi(k_1)),(\pi(k_2-1),\pi(k_2)),(\pi(k_3-1),\pi(k_3)),(\pi(k_4-1),\pi(k_4))$ with $k_i,1\leq i \leq 4$ odd such that $\pi(k_1-1) \in V_q, \pi(k_1)\in W_a$, $\pi(k_2-1) \in V_q, \pi(k_2)\in W_b$, $\pi(k_3-1) \in V_l, \pi(k_3)\in W_a$ and $\pi(k_4-1) \in V_l, \pi(k_4)\in W_b$. As the positions $(\pi(k_i-1),\pi(k_i),i=1,2,3,4$ are all distinct, there are four distinct letters that appear these four positions in $\boldsymbol{\omega}$. Now without loss of generality suppose, from left to right $(\pi(k_4-1),\pi(k_4))$ is the rightmost (among the four positions mentioned above) in $\boldsymbol{\omega}$. Then we see that as $\pi(k_4-1)\in V_l$ and $\pi(t_l)$ comes before $\pi(k_4-1)$, it cannot be chosen freely. Using a similar argument we can see that $\pi(k_4)$ also cannot be chosen freely. Also they have been chosen as generating vertices of three different letters that have appeared in the positions $(\pi(k_i-1),\pi(k_i)), 1 \leq i \leq 3$. Using Lemma \ref{S-words}, this is not possible as the letter at $(\pi(k_4-1),\pi(k_4))$ is different from the previous three letters. Thus the hypergraph $H(\sigma,\tau)$ does not have a cycle. Also it is evident by the construction that this hypergraph is unique, i.e., for every special symmetric word we get a unique hypergraph without cycles.\\

\noindent Conversely, suppose $H(\sigma,\tau)$ is a hypergraph with no cycle and $|\sigma|+|\tau|=b+1$. We form a word of length $2k$ from it in the following manner. Now $\sigma ,\tau \in \mathcal{P}(k)$. Let $\sigma=\{V_0,V_1,\ldots,V_r\}$ and $\tau=\{W_1,\ldots,W_{b-r}\}$ (as $|\sigma|+|\tau|=b+1$). Then we choose the even vertices $\pi(2i),0 \leq i\leq k-1$ from $\sigma$ and odd vertices $\pi(2i-1), 1 \leq i \leq k$ from $\tau$ and $\pi(i)=\pi(j)$ if $i$ and $j$ belong to the same block of $\sigma$ or $\tau$ (depending on $i$ and $j$ both being even or odd respectively). Thus we get a word $\boldsymbol{\omega}$ of length $2k$ whose even and odd generating vertices are $\{\pi(\min \{V_s\})\}_{0 \leq s \leq r}$ and $\{\pi(\min \{W_t\})\}_{1 \leq t \leq (b-r)}$ respectively. Thus there are $b$ distinct letters in $\boldsymbol{\omega}$. Now as $H(\sigma,\tau)$ does not have a cycle, using the same arguments as the previous paragraph, it can be shown that all the generating vertices can be chosen freely. This can happen only of the word is special symmetric. Thus we obtain $\boldsymbol{\omega} \in SS_b(2k)$ with $(r+1)$ even generating vertices. It is easy to see that two hypergaphs with no cycle cannot give rise to the same special symmetric word. 

Hence there is a one-one correspondence between special symmetric words and hypergraphs with no cycle. This completes the proof of this lemma.
\end{proof}

Thus we see that \eqref{Zak-limit} can be written as 
\begin{align}
\beta_{k}(\mu)& =\displaystyle \sum_{r=0}^{k-1}\sum_{\underset{\underset{\text{ even generating vertices}}{\text{having }(r+1)}}{\pi \in SS(2k)}} y^r C_{\pi} \nonumber \\
& = \displaystyle \sum_{r=0}^{k-1} \sum_{\underset{\text{with }r+1 \text{ blocks}}{\sigma\in \mathcal{P}(k)}} \sum_{\underset{H(\sigma,\tau) \text{ has no cycle}}{\tau \in \mathcal{P}(k)}} \prod_{i=1}^{b-r} y^{\frac{r}{b-r}} f(W_i)
\end{align} 
where $W_i$ are the blocks of $\tau$ and $f$ is some function determined by $(C_{2k})_{k\geq 1}$ (and not necessarily multiplicative in the sense of partitions.)

Thus using this and Remark \ref{unbounded support}, we conclude that our result generalises Theorem 3.2 in \citep{Benaych-Georges2012} and are the same when the entries satisfy \eqref{ck}.\\

\vskip3pt

In Proposition 3.1 in \citep{noiry2018spectral}, the author describes the limiting moments of the distribution via equivalence class of words (different from our notionof words) which we call \textit{Noiry words}. Let us first recall the description of Noiry words from Section 3 in \citep{noiry2018spectral}.\\
 
 \vskip3pt

\noindent \textbf{Noiry words}: Suppose $G=(V,E)$ is a labeled graph with labeling of the vertices. A word of length $k\geq 1$ on $G$ is a sequence of labels $i_1,i_2,\ldots,i_k$ such that for each $j\in \{1,2,\ldots,k-1\}$, $\{i_j,i_{j+1}\}$ is a pair of adjacent labels, i.e., the associated vetrices are neighbours in $G$. A word of length $k$ is closed if $i_1=i_k$. Such closed words will be called Noiry words.\\

\noindent \textbf{Equivalence of Noiry words}: Let $\boldsymbol{i}=i_1,i_2,\ldots,i_k$ and $\boldsymbol{i^{\prime}}=i_1^{\prime},i_2^{\prime},\ldots,i_k^{\prime}$ are two Noiry words on two labeled graphs $G$ and $G^{\prime}$. These words are said to be equivalent if there is a bijection $\sigma$ of $\{1,2,\ldots,|V|\}$ such that $\sigma(i_j)=i_j^{\prime},1 \leq j \leq k$. This defines an equivalence relation on the set of all Noiry words, thereby giving rise to equivalence classes of Noiry words. $\mathbf{W}_k(a,a+1,l,\boldsymbol{b}), \boldsymbol{b}=(b_1,b_2,\ldots,b_a)\in \mathbb{N}^a, b_i\geq 2, \displaystyle \sum_{i=1}^a b_i=2k$,( see Section 3 and equation 3.2 in \citep{noiry2018spectral}) is such an equiavlence class of Noiry words on a labeled rooted planar tree with $a$ edges, of which $l$ are odd and each edge is browsed $b_i$ times, $1 \leq i \leq a$. In the next lemma we show how each of these equivalence classes of words correspond to special symmetric words.

\begin{lemma}
Each equivalence class $\mathbf{W}_k(a,a+1,l,\boldsymbol{b}),\boldsymbol{b}=(b_1,b_2,\ldots,b_a)\in \mathbb{N}^a, b_i\geq 2, \displaystyle \sum_{i=1}^a b_i=2k$ is a word $\boldsymbol{\omega} \in SS_a(2k)$ with $l$ odd generating vertices and each letter appearing $b_i, 1\leq i \leq a$ times in $\boldsymbol{\omega}$. 
\end{lemma}
\begin{proof}
Recall from Section \ref{preliminaries} that we have defined words to be equivalence classes of circuits with the relation arising from the link functions (see \eqref{S-link}). Now Noiry words are not equivalence classes to begin with, they form equivalence classes if they are relabeled in a certain way as described above. From this and how we have defined equivalence of circuits, observe that an equivalence class of Noiry words are nothing but a word in our case.  Now the only words  with $a$ distinct letters for which $a+1$ generating vertices can be chosen freely are the special symmetric words with $a$ distinct letters (see Lemma \ref{S-words}). Thus $\mathbf{W}_k(a,a+1,l,\boldsymbol{b}),\boldsymbol{b}=(b_1,b_2,\ldots,b_a)\in \mathbb{N}^a, b_i\geq 2, \displaystyle \sum_{i=1}^a b_i=2k$ is a word $\boldsymbol{\omega} \in SS_a(2k)$ with $l$ odd generating vertices and each letter appearing $b_i, 1\leq i \leq a$ times in $\boldsymbol{\omega}$.
\end{proof}

Using this lemma it readily follows that 
\begin{align*}
\displaystyle\sum_{a=1}^k \sum_{l=1}^a \sum_{\underset{b_i\geq 2,b_1+\cdots+ b_{a}=2k}{\boldsymbol{b}=(b_1,b_2,\ldots,b_a)}}|\mathbf{W}_k(a,a+1,l,\boldsymbol{b})|  =  \displaystyle \sum_{a=1}^k \sum_{l=1}^{a}\sum_{\underset{\underset{\text{ with block sizes } b_1,\ldots,b_a }{\text{having }l \text{ odd generating vertices}}}{\pi \in SS_a(2k)}} 1 =  \displaystyle \sum_{l=1}^k \sum_{\underset{\text{having }l \text{ odd generating vertices}}{\pi \in SS(2k)}} 1.
\end{align*}
Hence it is easy to see \eqref{Zak-limit} is same as the expression of moments obtained in equation (3.2) in \citep{noiry2018spectral}.\\

\section{Appendix}\label{appendix}
\begin{lemma}
Suppose $\{x_{ij};1\leq i \leq p, 1 \leq j \leq j\leq n\}$ are independent variables that satisfy Assumption A and $y_{ij}= x_{ij}\boldsymbol{1}_{[|x_{ij}|\leq t_n]}$. Then 
\begin{itemize}
\item[(i)] $\displaystyle\frac{1}{p} \sum_{i,j} (y_{ij}^2-\mathbb{E}[y_{ij}^2]) \rightarrow 0  \ \ \text{ almost surely as } p \rightarrow \infty. $
\item[(ii)] Additionally if, $\frac{1}{p} \displaystyle \sum_{i,j} x_{ij}^2\boldsymbol{1}_{[|x_{ij}|> t_n]} \rightarrow 0$ almost surely (or in probability), $\limsup_p \displaystyle\frac{1}{p} \sum_{i,j} x_{ij}^2<\infty$ almost surely (or in probability).  
\end{itemize}

\end{lemma}
\begin{proof}
\begin{itemize}
\item[(i)] Let $\epsilon>0$ be fixed. Then 
\begin{align*}
\mathbb{P}\bigg[\big|\frac{1}{p}\displaystyle \sum_{i,j} (y_{ij}^2-\mathbb{E}[y_{ij}]^2)\big|>\epsilon\bigg]& \leq \frac{1}{\epsilon^4p^4} \mathbb{E}\bigg[ \big(\displaystyle \sum_{i,j} y_{ij}^2-\mathbb{E}[y_{ij}^2])\big)^4\bigg]\\
& = \frac{1}{\epsilon^4p^4} \mathbb{E}\bigg[ \displaystyle \sum_{\underset{j_1,j_2,j_3,j_4}{i_1,i_2,i_3,i_4}} \prod_{l=1}^4\big(y_{i_lj_l}^2 - \mathbb{E}[y_{i_lj_l}^2])\big)^4\bigg].
\end{align*}
As $y_{ij}$s are independent, the above inequality becomes
\begin{align*}
\mathbb{P}\bigg[\big|\frac{1}{p}\displaystyle \sum_{i,j} (y_{ij}^2-\mathbb{E}[y_{ij}]^2)\big|>\epsilon\bigg]& \leq \frac{1}{\epsilon^4 p^4} \sum_{ij} \mathbb{E}\big[(y_{ij}^2- \mathbb{E}[y_{ij}^2])^4\big] +\\ & \  6 \frac{1}{\epsilon^4 p^4} \sum_{\underset{j_1,j_2}{i_1,i_2}} \mathbb{E}\big[(y_{i_1j_1}^2- \mathbb{E}[y_{i_1j_1}^2])^2 (y_{i_2j_2}^2- \mathbb{E}[y_{i_2j_2}^2])^2\big].
\end{align*}
Now from \eqref{gkeven}, as $\{g_{2k,n}\}$ are bounded integrable, the first term in the rhs of the above inequality is $\mathcal{O}(\frac{1}{p^3})$ and the second term is $\mathcal{O}(\frac{1}{p^2})$. Therefore, 
\begin{align*}
\displaystyle \sum_p \mathbb{P}\bigg[\big|\frac{1}{p}\displaystyle \sum_{i,j} (y_{ij}^2-\mathbb{E}[y_{ij}]^2)\big|>\epsilon\bigg]< \infty.
\end{align*}
Hence by Borel-Cantelli lemma, $\displaystyle\frac{1}{p} \sum_{i,j} (y_{ij}^2-\mathbb{E}[y_{ij}^2]) \rightarrow 0 \ \ \ \text{ almost surely as } p \rightarrow \infty.$
\end{itemize}

\item[(ii)] Observe that $\displaystyle \sum_{i,j} x_{ij}^2= \sum_{i,j} \big(y_{ij}^2 +x_{ij}^2\boldsymbol{1}_{[|x_{ij}|> t_n]}\big)$. Then by the condition $\frac{1}{p}\displaystyle \sum_{i,j} x_{ij}^2\boldsymbol{1}_{[|x_{ij}|> t_n]} \rightarrow 0$ almost surely (or in probability) and (i), (ii) holds true. 
\end{proof}

%\begin{figure}[htp]
%\begin{subfigure}{.5\textwidth}
%  \centering
%  % include first image
%  \includegraphics[width=.7\linewidth]{wignerusual.pdf}  
%  \caption{Input is i.i.d $N(0,1)/\sqrt{n}$.}
%  %\label{fig:sub-first}
%\end{subfigure}
%\begin{subfigure}{.5\textwidth}
%  \centering
%  % include second image
%  \includegraphics[width=.7\linewidth]{wignerbern-2.pdf}  
%  \caption{Input is i.i.d Ber$(2/n)$ for every $n$.}
%  %\label{fig:sub-second}
%\end{subfigure}
%
%\caption{Histogram of the eigenvalues of $W_n$ for $n=1000, 30$ replications.}
%% Top left: i.i.d $N(0,1)/\sqrt{n}$; top right: i.i.d Ber$(3/n)$ for every $n$; bottom left: $x_{k,n}= \sin(\frac{k\pi}{n})\mbox{Ber}(3/n)$ for every $n$; bottom right: $x_{k,n}=(k/n)\mbox{Ber}(1/k)$ for every $n$}
%\label{fig:revfig2}
%\end{figure}

Here are some simulations of the Sample covariance matrices with variance profiles.

\begin{figure}[htp]

 \centering
\includegraphics[width=.7\linewidth]{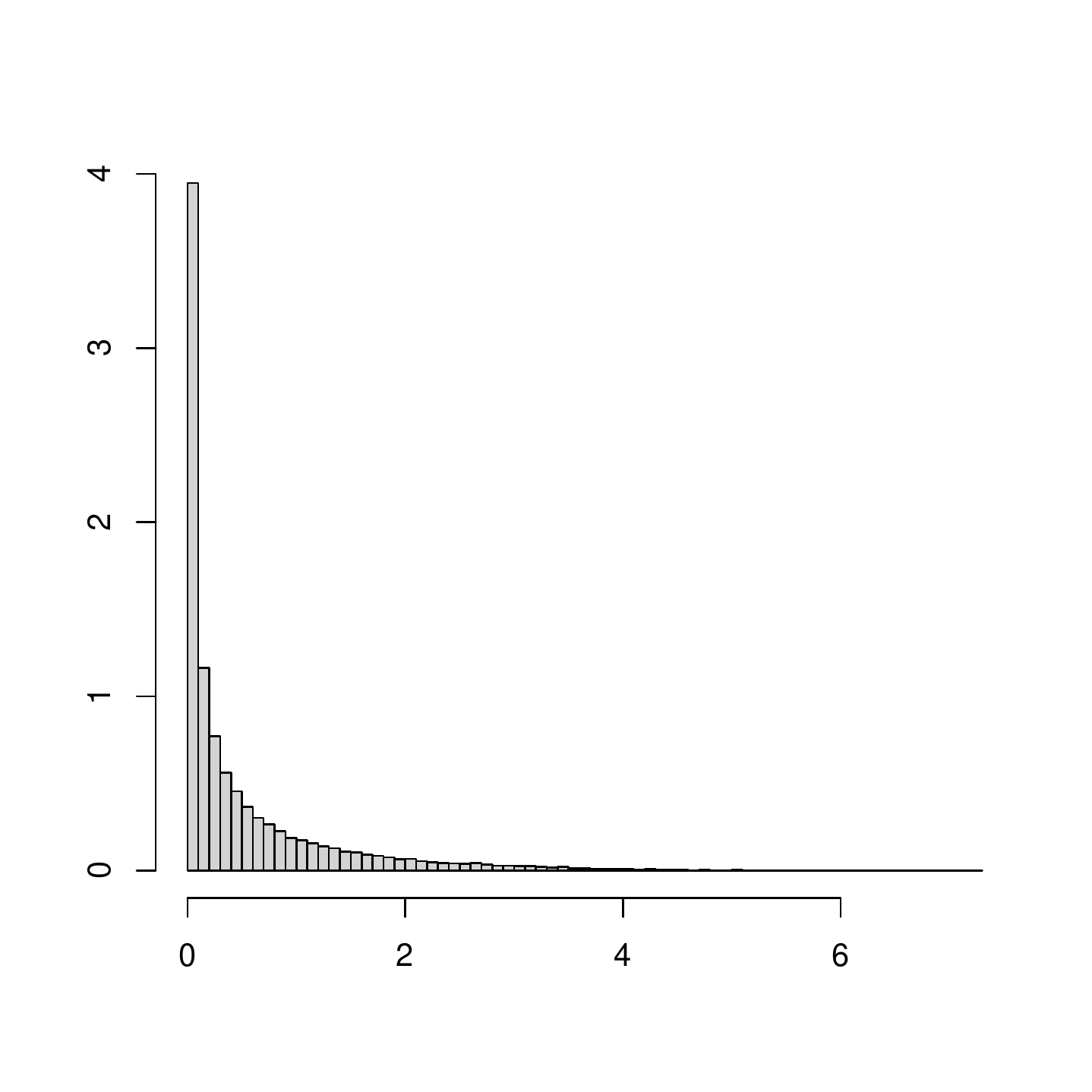}  
  \caption{Input is i.i.d $x_{ij}=\frac{(i+j)^2}{2n^2}Ber(3/n)$ for every $n$.}
  %\label{fig:sub-second}
 % \caption{Histogram of the eigenvalues of $X_pX_p^T$ for $p=500,n=1000, 30$ replications.}
  \end{figure}
\begin{figure}
\centering
  % include second image
  \includegraphics[width=.7\linewidth]{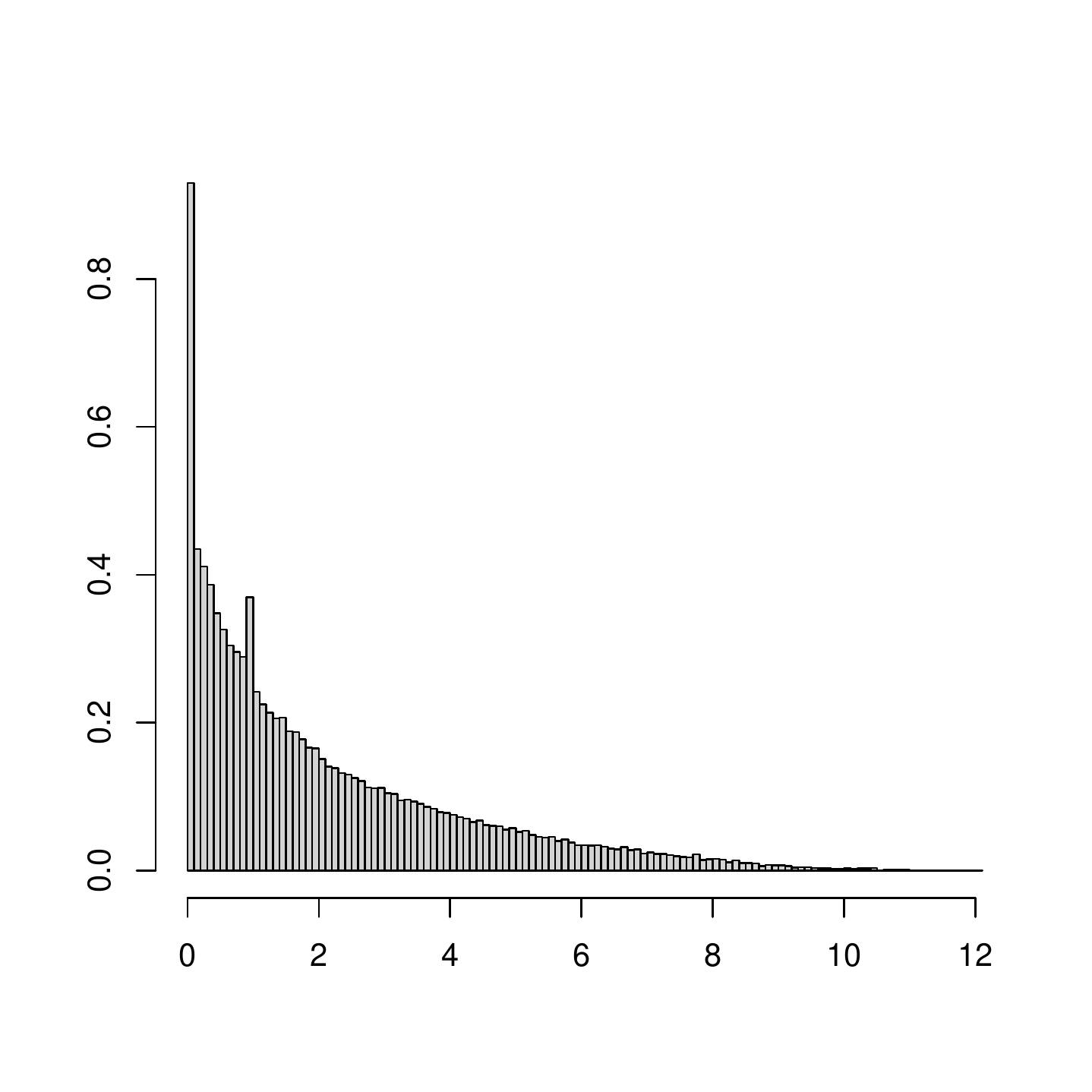}  
  \caption{Input is i.i.d $x_{ij}=\sin(\frac{\pi(i+j)}{2n})Ber(3/n)$ for every $n$.}
  %\label{fig:sub-second}
  \caption*{Histogram of the eigenvalues of $X_pX_p^T$ for $p=500,n=1000, 30$ replications.}
\end{figure}

\providecommand{\bysame}{\leavevmode\hbox to3em{\hrulefill}\thinspace}
\providecommand{\MR}{\relax\ifhmode\unskip\space\fi MR }
\providecommand{\MRhref}[2]{%
  \href{http://www.ams.org/mathscinet-getitem?mr=#1}{#2}
}
\providecommand{\href}[2]{#2}

\bibliographystyle{plainnat}
\bibliography{mybibfilefinal}

\end{document}